%% file: BoundednessInStrictDQ.tex
\documentclass[12pt,a4paper,final]{article}
\usepackage{geometry}
\usepackage[english, strings]{babel}
\usepackage{nchairx}
\usepackage[utf8]{inputenc}
\usepackage[T1]{fontenc}       
\usepackage{longtable}         
\usepackage{exscale}           
\usepackage[sort]{cite}        
\usepackage{xspace}            
\usepackage{tikz}              
\usetikzlibrary{cd, decorations.pathmorphing}
\usepackage{ifdraft}           
\usepackage[expansion=false    
]{microtype}        

\usepackage[dvipsnames]{xcolor} 
\usepackage[backref=page,      
final=true,         
pdfpagelabels,       
ocgcolorlinks, 
citecolor = ForestGreen,
]{hyperref}         
\usepackage{ocgx2} 

\usepackage{parskip}
\input{Auxiliary/Macros.tex} 

\author{Michael Heins}
\title{ \vspace{-1cm}
    \sc Boundedness in Strict \\ Deformation Quantization}
\date{}

\addto\captionsenglish{
    \renewcommand{\contentsname}%
    {Table of Contents}%
}

\author{\AuthorOne \\ \AuthorAddressOne}

\begin{document}

\maketitle

\vspace{-0.5cm}

\begin{abstract}
    This paper provides a complex-analytic interpretation of certain
    topologies on the tensor algebra, which are a central tool within
    strict deformation quantization. We prove that the symmetric tensor
    algebra $\Sym(V)$ of a barelled nuclear DF-space $V$ endowed with
    any of these topologies may be understood as an algebra of bounded
    Fréchet holomorphic functions on the strong dual space $V'_\beta$.
    The coarsest of the topologies then corresponds to the topology of
    uniform convergence on bounded subsets. The finer ones provide an
    infinite-dimensional generalization of the order of an entire
    holomorphic function. To facilitate these results, we study the
    interplay between boundedness and continuity of holomorphic mappings
    between locally convex spaces. As a byproduct, we establish a simple
    sufficient criterion for Fréchet holomorphy as well as several
    completeness results for spaces of bounded functions.
\end{abstract}

\tableofcontents

\renewcommand{\thefootnote}{\fnsymbol{footnote}}
\footnotetext{\textbf{MSC classication:} 46G20, 46G25, 53D55}
\renewcommand{\thefootnote}{\arabic{footnote}}

\newpage

\section{Introduction}
\label{sec:Introduction}%
\input{TeX/Intro.tex}

\section*{Acknowledgements}
The author would like to thank Stefan Waldmann for numerous fruitful
discussions on the topic, and for pointing out the class of examples
discussed in Example~\ref{ex:NuclearFrechet}. Moreover, the author is
indebted to Matthias Sch\"otz for proposing the setting of
Example~\ref{ex:HolomorphicUnbounded} as a toy-model for the space of
compactly supported smooth functions. Last but not least, the great
patience afforded by Johanna Fladung and Bas Janssens has resulted in a
vastly improved introduction. Exploring the interplay between
complex analysis and strict deformation quantization was at the heart of
the author's thesis \cite{heins:2024a}. This article constitutes a
refinement of the second chapter thereof and is partly funded by the
project \emph{Regularity in the representation theory and deformation
quantization of loop groups}, which is financed by the Dutch Research
Council (NWO) under the grant
\href{https://doi.org/10.61686/NZDJB53943}{doi.org/10.61686/NZDJB53943}.

\newpage

\section{Spaces of Bounded Functions}
\label{sec:SpacesOfBoundedFunctions}%

This section concerns the locally convex analysis of bounded holomorphic
functions. Most of what we present seems to be well-known, if a bit
fallen out of favour. Likewise, many of our results may be alternatively
derived from Mackey's theory of vector bornologies \cite{mackey:1946a},
a comprehensive exposition of which is the monograph
\cite{hogbe-nlend:1977a}. Our presentation is more analytic in nature
and our techniques are tailored for the problems we are going to face
within Section~\ref{sec:TensorsAsPolynomials}. As a byproduct, we derive
a simple sufficient criterion for Fréchet holomorphy on Baire spaces in
Proposition~\ref{prop:FrechetBaire}, which might be of independent
interest.

\subsection{Bounded Mappings}
\label{subsec:FunctionsBounded}%
\input{TeX/Bounded.tex}

\subsection{Polynomials}
\label{subsec:FunctionsPolynomials}%
\input{TeX/Polynomials}

\subsection{Fréchet Holomorphic Functions}
\label{subsec:FunctionsFrechet}%
\input{TeX/Frechet}

\section{Projective and Injective $R$-topologies}
\label{sec:TensorAlgebras}%

In this section, we briefly review projective and injective tensor
products as well as the topologies they induce on the tensor
algebra. The projective variant was first used
in~\cite{beiser.waldmann:2014a} and subsequently
systematically studied and utilized for further strict deformations
within~\cite{waldmann:2014a}. The injective variant is new albeit
analogous, and its relevance from the point of view of holomorphicity
will be made clear within Theorem~\ref{thm:IotaEmbedding1}. In both
cases, the principal idea is that we want to endow the tensor algebra
\begin{equation}
    \Tensor^\bullet(V)
    \coloneqq
    \bigoplus_{n=0}^\infty
    \Tensor^n(V)
\end{equation}
of some locally convex space $V$ with the structure of a locally convex
algebra. Taking a step back, this first of all amounts to a choice of a
topology on the tensor powers
\begin{equation}
    \Tensor^n(V)
    \coloneqq
    \underbrace{V \tensor \cdots \tensor V}_{\textrm{$n$-times}}
\end{equation}
for all $n \in \N_0$. This is a surprisingly subtle enterprise, as one
can make numerous, typically inequivalent choices, a problem famously
explored by Grothendieck in his doctoral thesis, the results of which he
published as \cite{grothendieck:1955a}.

\subsection{Tensor Products of Locally Convex Spaces}
\label{subsec:TensorProducts}%
\input{TeX/TensorProducts}

\subsection{$R$-Topologies}
\label{subsec:RTopologies}%
\input{TeX/RTopologies}

\section{Tensors as Polynomials}
\label{sec:TensorsAsPolynomials}%
\input{TeX/TensorsAsPolynomials.tex}

\bibliographystyle{nchairx}
\phantomsection
\addcontentsline{toc}{section}{Bibliography}
\bibliography{Auxiliary/Boundedness,dqbook,dqarticle,preprints}

\end{document}

%% file: Auxiliary/Macros.tex
\newcommand{\AuthorOne}{\textbf{Michael Heins}}

\newcommand{\AuthorAddressOne}{
    \begin{minipage}{10cm}
        \centering\scriptsize
        Institute of Applied Mathematics \\
        Delft University of Technology \\
        Mekelweg 4 \\
        2628 CD Delft \\
        Netherlands \\
        \texttt{michael[dot]heins[dot]mathematics[at]gmail[dot]com}
    \end{minipage}
    \\
}

\makeatletter

\newcommand{\bibnote}[2]{\nocite{#1}\@namedef{#1chairxnote}{#2}}
\makeatother

\newcommand{\C}{\field{C}}

\newcommand{\R}{\field{R}}

\newcommand{\N}{\field{N}}

\newcommand{\Z}{\field{Z}}

\renewcommand{\epsilon}{\varepsilon}

\DeclareFontFamily{U}{mathx}{}
\DeclareFontShape{U}{mathx}{m}{n}{ <-> mathx10 }{}
\DeclareSymbolFont{mathx}{U}{mathx}{m}{n}
\DeclareFontSubstitution{U}{mathx}{m}{n}
\DeclareMathAccent{\widecheck}{0}{mathx}{"71}

\renewcommand{\Holomorphic}{\mathscr{H}}

\newcommand{\symmetric}{{\operatorname{s}}}

\newcommand{\cs}{\operatorname{cs}}

\newcommand{\polar}{\textrm{\tiny \fontencoding{U}\fontfamily{ding}\selectfont\symbol{'136}}}

\newcommand{\Bdd}{\mathfrak{B}}

\DeclareFontFamily{U}{mathx}{}
\DeclareFontShape{U}{mathx}{m}{n}{ <-> mathx10 }{}
\DeclareSymbolFont{mathx}{U}{mathx}{m}{n}
\DeclareFontSubstitution{U}{mathx}{m}{n}
\DeclareMathAccent{\widecheck}{0}{mathx}{"71}

%% file: TeX/Intro.tex

It is undeniable that the pursuit and study of formal
star products, as they were originally conceived in
\cite{bayen.et.al:1977a, bayen.et.al:1978a} has led to a wealth of
mathematical insights and applications. However, from the vantage point
of the original motivation, namely quantization within mathematical
physics, the theory harbours a serious flaw: the formal parameter
$\hbar$ carries the interpretation of Planck's constant. As such it has a
definitive size after choosing a system of units. Thus, given a formal
deformation $(\algebra{A}\formal{\hbar},\star)$ of a Poisson algebra
$(\algebra{A},\cdot,\{\argument,\argument\})$, one is invariably led to
study the convergence of the $\algebra{A}$-valued formal power series
\begin{equation}
    \star
    \colon
    \algebra{A}\formal{\hbar}
    \times
    \algebra{A}\formal{\hbar}
    \longrightarrow
    \algebra{A}\formal{\hbar}
\end{equation}
given by
\begin{equation}
    \label{eq:FormalDeformation}
    a \star b
    =
    a \cdot b
    +
    \sum_{n=1}^{\infty}
    D_n(a,b) \cdot \hbar^n,
\end{equation}
where the $D_n$ are $\C\formal{\hbar}$-linear extensions of bilinear
mappings defined on $\algebra{A}$. It is the central objective of strict
deformation quantization to address this issue.

In recent years, there have been at least two fruitful schools of
thought regarding this convergence problem. One approach, which was
initiated by \cite{rieffel:1989a, rieffel:1993a} in a
{C$^*$-algebraic} framework, is to construct integral formulas. More
recent
advances in this direction such as \cite{bieliavsky.massar:2001b,
bieliavsky:2002a, bieliavsky.bonneau.maeda:2007a, landsman:1998a} have
moreover been based on careful use of oscillatory Riemann integrals.
Such theories are inherently global in nature and governed by
sufficiently rapid decay of the observables, allowing for the
construction of C$^*$-norms. This naturally leads to mathematically
interesting spaces of functions and symbols generalizing the Schwartz
space. Under favourable circumstances, one may then recover the formal
star product~\eqref{eq:FormalDeformation} as an asymptotic
expansion of the oscillatory integrals at $\hbar = 0$. That is to say,
while the series itself is not convergent, its partial sums nevertheless
approximate the correct objects. One notable example is the
quantization of K\"{a}hler manifolds \cite{cahen.gutt.rawnsley:1990a,
    cahen.gutt.rawnsley:1993a, cahen.gutt.rawnsley:1994a,
    cahen.gutt.rawnsley:1995a} by means of Berezin-Toeplitz operators
\cite{berezin:1975a, berezin:1975b, berezin:1975c}, which ultimately
could even be refined to a continuous field of C$^*$-algebras within
\cite{bordemann.meinrenken.schlichenmaier:1994a}. While these techniques
are certainly powerful, there is in general no universal way to
produce star products by means of oscillatory integrals.

Another, and in many regards orthogonal, approach introduced
by \cite{waldmann:2014a} and motivated by the earlier work
\cite{omori.maeda.miyazaki.yoshiaki:1999a} consists in endowing
$\algebra{A}$ with a locally convex topology admitting for the
convergence of the series~\eqref{eq:FormalDeformation}. In practice, it
has proven fruitful to take a step back and even establish the
continuity of the star product as a bilinear mapping. This often comes
with the additional benefit of obtaining a holomorphic dependence
on~$\hbar$ as a byproduct of the required estimates. Thus, the
resulting theory is somewhat local and very much complex-analytic in
nature. The precise strategy is then as follows:

The data consists in a formal deformation $\star$ of a Poisson algebra
$\algebra{A}$. As a first step, one then looks for a Poisson subalgebra
$\algebra{A}_{\conv} \subseteq \algebra{A}$, on which the series
\eqref{eq:FormalDeformation} converges for obvious reasons, say by
terminating after finitely many terms. Such subalgebras tend to be too
small to be interesting from the point of view of both analysis and
physics. For instance, in many examples, they fail to contain the
relevant Hamiltonian functions. To remedy this, one then endows
$\algebra{A}_{\conv}$ with a topology such that $\star$ is continuous as
a bilinear mapping. This then has the pleasant consequence that star
product uniquely extends to a multiplication law on the completion
$\widehat{\algebra{A}}_{\conv}$, which now hopefully contains a wealth
of interesting observables. This programme is our principal motivation.
Although there is no general theory as of yet, there is a growing list
of examples
\cite{heins:2024a, heins.roth.waldmann:2023a,
    heins.moucha.roth.sugawa:2025a,
    kraus.roth.schleissinger.waldmann:2023a:pre2,
    esposito.schmitt.waldmann:2019a, kraus.roth.schoetz.waldmann:2019a,
    schoetz.waldmann:2018a, esposito.stapor.waldmann:2017a,
    beiser.waldmann:2014a, schmitt.schoetz:2022a,
    barmeier.schmitt:2022a, schmitt:2021a, pirkovskii:2025a,
    esposito.heins.waldmann:2026a:pre}, and
there are quite a few recurring themes throughout. The survey
\cite{waldmann:2019a} provides a more systematic overview.

Of course, one wants to retain as much algebraic structure as possible
in the analytic regime. For instance, if $\algebra{A}$ consists of
functions, then a natural demand is the continuity of the point
evaluations. In this case, the abstract completion
$\widehat{\algebra{A}}_{\conv}$ may itself be viewed as an algebra of
functions. The prototypical example consists in $\algebra{A}
\coloneqq \Cinfty(\R^{2d})$ with the standard symplectic Poisson bracket
\begin{equation}
    \label{eq:StandardPoisson}
    \{f,g\}
    \coloneqq
    \sum_{k=1}^{d}
    \frac{\partial f}{\partial q^k}
    \cdot
    \frac{\partial g}{\partial p_k}
    -
    \frac{\partial g}{\partial q^k}
    \cdot
    \frac{\partial f}{\partial p_k}
\end{equation}
and the standard-ordered star product
\begin{equation}
    \label{eq:StandardOrder}
    f \star_\std g
    \coloneqq
    f \cdot g
    +
    \sum_{n=1}^{\infty}
    \frac{\hbar^n}{n!}
    \sum_{k_1,\ldots,k_n=1}^{d}
    \frac{\partial^n f}{\partial p_{k_1} \cdots \partial p_{k_n}}
    \cdot
    \frac{\partial^n g}{\partial q^{k_1} \cdots \partial q^{k_n}},
\end{equation}
where $(q^1, \ldots, q^d, p_1,\ldots,p_d) \subseteq \R^{2d}$ is a vector
space basis. By virtue of the classical Borel Lemma~\cite{borel:1895a},
pointwise convergence of \eqref{eq:StandardOrder} for generic $f,g \in
\Cinfty(\R^{2d})$ is hopeless if $\hbar \neq 0$. That being said, the
series terminates on the polynomial algebra
\begin{equation}
    \label{eq:PolynomialsFiniteDimensional}
    \algebra{A}_{\conv}
    \coloneqq
    \C
    [q^1,\ldots,q^d,
    p_1,\ldots,p_d],
\end{equation}
viewed as a subalgebra of $\Cinfty(\R^{2d})$ by passing to polynomial
functions. Treating $q$ and $p$-directions independently, we view
$\C[q^1,\ldots,q^d] \subseteq \Holomorphic(\C^d)$ as entire holomorphic
functions. Recall that $f \in \Holomorphic(\C^d)$ is said to be of
finite order at most $\varrho > 0$ and of minimal type if
\begin{equation}
    \label{eq:SeminormsFiniteOrder}
    \norm{f}_\epsilon^{(\varrho)}
    \coloneqq
    \sup_{z \in \C^d}
    \exp
    \bigl(
        - \epsilon \norm{z}^{\varrho}
    \bigr)
    \cdot
    \abs[\big]
    {f(z)}
    <
    \infty
\end{equation}
for all $\epsilon > 0$. The expressions
$\norm{\argument}_{\epsilon}^{(\varrho)}$
constitute (inequivalent) norms and thus define a locally convex
topology on
\begin{equation}
    \Holomorphic_\varrho(\C^d)
    \coloneqq
    \bigl\{
        f \in \Holomorphic_\varrho(\C^d)
        \colon
        \norm{f}_\epsilon^{(\varrho)}
        <
        \infty
        \textrm{ for all }
        \epsilon > 0
    \bigr\}.
\end{equation}
In \cite{omori.maeda.miyazaki.yoshiaki:1999a} it was shown that endowing
\eqref{eq:PolynomialsFiniteDimensional} with the subspace topology
inherited from
\begin{equation}
    \Holomorphic_{2}(\C^d)
    \times
    \Holomorphic_{2}(\C^d)
\end{equation}
leads to the continuity of \eqref{eq:StandardOrder}. Following the
general strategy, $\star_\std$ extends to a continuous product on the
completion
\begin{equation}
    \widehat{\algebra{A}}_{\conv}
    =
    \Holomorphic_{2}(\C^d)
    \times
    \Holomorphic_{2}(\C^d).
\end{equation}
Unfortunately, the space~$\Holomorphic_2(\C^d)$ does not contain
Gaussians, which however arise in typical Hamiltonians, see
\cite[Sec.~3.2 \& 4.1]{esposito.heins.waldmann:2026a:pre} for a partial
remedy based on analyic vectors.

Either way, seeking the convergence of \eqref{eq:StandardOrder}, one is
naturally led to a well-studied object from complex analysis. Returning
to general $\varrho > 0$ and carefully applying the Cauchy estimates,
one may pass to the equivalent set of seminorms given by
\begin{equation}
    \label{eq:SeminormsFiniteOrderTaylor}
    \seminorm{p}_r^{(\varrho)}(f)
    \coloneqq
    \abs{f(0)}
    +
    \sum_{n=1}^\infty
    \frac{r^n}{n!}
    \cdot
    n!^{1/\varrho}
    \cdot
    \sum_{k_1,\ldots,k_n=1}^d
    \abs[\bigg]
    {
        \frac{\partial^n f}{\partial z^{k_1} \cdots \partial z^{k_n}}(0)
    }
\end{equation}
indexed by $r \ge 0$, see e.g. \cite[Sec.~5.4]{waldmann:2014a} for a
detailed proof. Being able to characterize finite order either by global
growth \eqref{eq:SeminormsFiniteOrder} or Taylor
data~\eqref{eq:SeminormsFiniteOrderTaylor} allows for great flexibility.
For instance, proving the continuity of pointwise multiplication is
trivial in the former defining system of seminorms, and a bit of work in
the latter. Likewise, the opposite is true for differentiation, and by
extension the star product itself.

It is the purpose of this paper to extend this conceptually pleasing
relationship into the realm of infinite-dimensional holomorphy. Taking a
step back, we consider the following abstract setup. Let $V$ be a
locally convex space and~$v_1,\ldots,v_n \in V$. Then
\begin{equation}
    \label{eq:PolynomialPrototype}
    V'
    \ni
    v'
    \quad \mapsto \quad
    v'(v_1) \cdots v'(v_n)
    \in
    \C
\end{equation}
defines a polynomial mapping on the dual space $V'$. This gives rise to a
linear map
\begin{equation}
    \label{eq:CanonicalEmbeddingIntro}
    \iota
    \colon
    \Sym^\bullet(V)
    \coloneqq
    \bigoplus_{n=0}^\infty
    \Sym^n(V)
    \longrightarrow
    \Pol(V'),
\end{equation}
which interprets elements of the symmetric tensor algebra
$\Sym^\bullet(V)$ of $V$ as polynomials on its dual space $V'$. In the
above example, i.e. $V = \C^{d}$, choosing a basis
$\basis{e}_1,\ldots,\basis{e}_d$ and using $V' \cong V$ by means of the
dual basis, the map \eqref{eq:CanonicalEmbeddingIntro} is simply given by
\begin{equation}
    \iota
    \bigl(
        \basis{e}_{k_1} \vee \cdots \vee \basis{e}_{k_n}
    \bigr)
    \at[\Big]{z}
    =
    z^{k_1} \cdots z^{k_n}
    \qquad
    \textrm{for }
    k_1, \ldots, k_n
    =
    1, \ldots, d.
\end{equation}
For now, this warrants our notation in
\eqref{eq:CanonicalEmbeddingIntro}, which we will address properly in
Section~\ref{subsec:FunctionsPolynomials}.

Mimicking
\eqref{eq:SeminormsFiniteOrderTaylor} and using the projective tensor
product, \cite[Def.~3.5]{waldmann:2014a} introduced a
locally convex topology on $\Sym^\bullet(V)$, again depending on a
parameter $\varrho > 0$. We will make this precise in
\eqref{eq:SeminormRTopology}, where then $R = 1/\varrho$ provides a
bridge to our earlier discussion. As for \eqref{eq:StandardOrder},
\cite[Lem.~3.10]{waldmann:2014a} then uses these topologies to provide
strict deformations of constant Poisson structures induced by continuous
bilinear mappings.

Returning to \eqref{eq:CanonicalEmbeddingIntro}, this article deals with
the finer structure of $\iota$. In
Theorem~\ref{thm:TensorsAsPolynomials}, we
describe properties of the polynomial functions in its image. They
turn out to carry more structure: most importantly, they are bounded,
i.e. map bounded sets to bounded sets. Moreover, having obtained a
useful topology for its domain, one would like to find a matching
topology for the polynomial functions $\Pol(V')$ such that $\iota$
becomes continuous. Again, boundedness provides the answer in the form
of the associated topology of uniform convergence on bounded sets. It
constitutes the appropriate generalization of the topology of locally
uniform convergence, i.e. the natural Fréchet topology of
spaces of holomorphic functions in finite dimesions. This adresses the
absence of local compactness.

This, in turn, raises the question of how one may describe the image of
the unique linear extension of $\iota$ to the completion of
$\Sym^\bullet(V)$. This leads beyond polynomials and into the realm of
bounded Fréchet holomorphic functions $\Holomorphic(V')$. Indeed, under
suitable assumptions $\iota$ becomes a homeomorphism onto its image,
and the completion of its image may be described as the closure within
$\Holomorphic(V')$. This is the content of
Theorem~\ref{thm:IotaEmbedding1} and Theorem~\ref{thm:IotaEmbedding2}
for two different types of tensor
products. Crucially, this strategy however requires the completeness
of the ambient space $\Holomorphic(V)$, which is the content of
Theorem~\ref{thm:CompletenessBoundedFrechet}, and seems to be new. To
fill the abstract theory with some life, we describe a large class of
locally convex spaces in Example~\ref{ex:NuclearFrechet}, to which our
results apply.

Finally, while many observable algebras for strict deformations will not
be completions of symmetric algebras directly, a number of constructions
have been inspired by its topologies. For instance, in
\cite{heins.roth.waldmann:2023a} the notion of $R$-entire functions on
Lie groups $G$ was introduced, which facilitated the continuity of the
Lie-theoretic generalization of the standard ordered star product from
\eqref{eq:StandardOrder}. Using the theory of analytic continuations and
a left-invariant version of the Cauchy estimates, \cite{heins:2025a:pre}
identified the entire functions as restrictions of globally defined
holomorphic functions on Hochschild's universal complexification~$G_\C$
of~$G$. Moreover, \cite[Rem.~4.16]{heins.roth.waldmann:2023a} provides
an interpretation of the $R$-topologies as entire vectors for the
translation representation. Expanding on these ideas has led to the
central techniques of \cite{esposito.heins.waldmann:2026a:pre}.
Similarly, \cite[Sec.~6]{heins.moucha.roth.sugawa:2025a}
provides a holomorphic interpretation of the observable algebras devised
within \cite{kraus.roth.schoetz.waldmann:2019a} as entire functions in a
distinguished chart.

The paper is structured as follows. We begin by studying the topology of
uniform convergence on bounded sets from a complex-analytic point of
view in Section~\ref{sec:SpacesOfBoundedFunctions}. More precisely,
Section~\ref{subsec:FunctionsBounded} concerns the space of all bounded
functions, Section~\ref{subsec:FunctionsPolynomials} deals with
polynomials and Section~\ref{subsec:FunctionsFrechet} with bounded
Fréchet holomorphic mappings. We proceed carefully, and provide some
counterexamples, as the literature is somewhat sparse and sometimes
contradictory on these topics. Afterwards, we briefly review the most
important properties of projective and injective tensor products, as
well as the associated $R$-topologies in
Section~\ref{sec:TensorAlgebras}. Finally,
Section~\ref{sec:TensorsAsPolynomials} is where both worlds meet. Here,
we state and prove our aforementioned main results
Theorem~\ref{thm:TensorsAsPolynomials}, Theorem~\ref{thm:IotaEmbedding1}
and Theorem~\ref{thm:IotaEmbedding2}.

%% file: TeX/Bounded.tex

Recall that a subset $B$ of a locally convex space $V$ is called bounded
if
\begin{equation}
    \sup_{v \in B}
    \seminorm{q}(v)
    <
    \infty
\end{equation}
for all continuous seminorms $\seminorm{q} \in \cs(V)$ defined on $V$.
Our interest lies not in the bounded sets themselves, but rather in
the functions respecting boundedness, which in turn are
then called bounded themselves. If $U \subseteq V$ is some subset, $W$
another locally convex space and $f \colon U \longrightarrow W$, then we
may rephrase this as the finiteness of the seminorms
\begin{equation}
    \label{eq:SeminormBounded}
    \seminorm{p}_{B,\seminorm{q}}(f)
    \coloneqq
    \sup_{v \in B}
    \seminorm{q}
    \bigl(
        f(v)
    \bigr),
\end{equation}
where $B \subseteq U$ is varied through the bounded subsets of $U$ and
$\seminorm{q}$ through $\cs(W)$. If $W = \C$ and $\seminorm{q} =
\abs{\argument}$, we simply write $\seminorm{p}_B \coloneqq
\seminorm{p}_{B,\abs{\argument}}$. Later, we will also use the
seminorms \eqref{eq:SeminormBounded} for not necessarily bounded sets,
denoted by the same symbol.

The most prominent example of bounded functions are the continuous linear
mappings. If the domain is bornological, boundedness conversely implies
continuity, which lets us encode continuity by means of an estimate in
the familiar manner. In Proposition~\ref{prop:PolynomialsBoundedness},
we shall see that these observations extend to higher order polynomials
defined on first countable locally convex spaces, and
Example~\ref{ex:DiscontinuousBoundedPolynomial} shows
that this may well fail beyond this situation.
Taking a step back, we are led to consider the locally convex space
\begin{equation}
    \label{eq:BoundedFunctions}
    \Bounded(U,W)
    \coloneqq
    \bigl\{
        f
        \colon
        U \longrightarrow W
        \; \big| \;
        \forall_{B \subseteq U \textrm{ bounded}, \,
        \seminorm{q} \in \cs(W)}
        \colon
        \seminorm{p}_{B,\seminorm{q}}(f)
        <
        \infty
    \bigr\}
\end{equation}
of all bounded functions defined on some subset $U \subseteq V$ with
values in another locally convex space $W$. We endow
\eqref{eq:BoundedFunctions} with the locally convex topology
arising from the seminorms \eqref{eq:SeminormBounded}. We call the
resulting topology the topology of uniform convergence on bounded sets or
$\beta$-topology. The proof of the following proposition confirms that
this is indeed appropriate terminology.
\begin{proposition}[Completeness of $\Bounded$]
    \label{prop:BoundedCompleteness}%
    Let $V,W$ be locally convex spaces and $U \subseteq V$ be
    a subset. If $W$ is (sequentially) complete Hausdorff, then so is
    the space~$\Bounded(U,W)$.
\end{proposition}
\begin{proof}
    Let $(f_\alpha)_{\alpha \in J} \subseteq \Bounded(U,W)$ be a
    Cauchy net. As finite sets are always bounded, the topology of
    pointwise convergence is coarser than the topology of uniform
    convergence on bounded subsets. Consequently, $\Bounded(U,W)$
    is Hausdorff and the net $(f_\alpha)$ converges pointwisely to
    its pointwise limit
    \begin{equation*}
        f
        \colon
        U \longrightarrow W, \quad
        f(v)
        \coloneqq
        \lim_{\alpha \in J}
        f_\alpha(v).
    \end{equation*}
    We prove that this convergence is actually uniform on bounded subsets of $U$.
    To this end, fix $\epsilon > 0$, $\seminorm{q} \in \cs(W)$ and a bounded
    subset $B \subseteq U$. As $(f_\alpha)_{\alpha \in J} \subseteq
    \Bounded(U,W)$ is a Cauchy net, we find an index $\alpha_0 \in J$
    such that
    \begin{equation*}
        \sup_{v \in B}
        \seminorm{q}
        \bigl(
        f_\alpha(v) - f_\beta(v)
        \bigr)
        =
        \seminorm{p}_{B,\seminorm{q}}
        \bigl(
        f_\alpha - f_\beta
        \bigr)
        \le
        \epsilon
        \qquad
        \textrm{for all }
        \alpha, \beta \later \alpha_0.
    \end{equation*}
    Consequently, the continuity of $\seminorm{q} \colon W
    \longrightarrow \R$ and the pointwise convergence at $v$ implies
    \begin{equation*}
        \seminorm{q}
        \bigl(
        f(v) - f_\beta(v)
        \bigr)
        =
        \lim_{\alpha \in J}
        \seminorm{q}
        \bigl(
        f_\alpha(v) - f_\beta(v)
        \bigr)
        \le
        \epsilon
        \qquad
        \textrm{for all }
        v \in B
        \textrm{ and }
        \beta \later \alpha_0.
    \end{equation*}
    Taking the supremum over $v \in B$ yields
    \begin{equation}
        \label{eq:BoundedCompletenessProof}
        \seminorm{p}_{B,\seminorm{q}}
        \bigl(
        f
        -
        f_\beta
        \bigr)
        \le
        \epsilon
        \qquad
        \textrm{for all }
        \beta \later \alpha_0.
        \tag{$\diamondsuit$}
    \end{equation}
    This establishes the uniform convergence $f_\alpha \rightarrow
    f$ on bounded subsets of $V$. It remains to show that $f$ is
    bounded itself. To prove this, keep the objects as above, say
    for the particular choice $\epsilon \coloneqq 1$. Then
    \eqref{eq:BoundedCompletenessProof} gives
    \begin{equation*}
        \sup_{v \in B}
        \seminorm{q}
        \bigl(
        f(v)
        \bigr)
        =
        \seminorm{p}_{B,\seminorm{q}}
        \bigl(
        f
        \bigr)
        \le
        \seminorm{p}_{B,\seminorm{q}}
        \bigl(
        f - f_{\alpha_0}
        \bigr)
        +
        \seminorm{p}_{B,\seminorm{q}}
        \bigl(
        f_{\alpha_0}
        \bigr)
        \le
        1
        +
        \seminorm{p}_{B,\seminorm{q}}
        \bigl(
        f_{\alpha_0}
        \bigr)
        <
        \infty
    \end{equation*}
    by boundedness of $f_{\alpha_0}$. Varying $B$ and $\seminorm{q}$ proves the
    boundedness of $f$, which establishes the completeness of
    $\Bounded(U,W)$. Replacing the nets with sequences throughout
    our considerations allows to infer the inheritance of
    sequential completeness in the same manner.
\end{proof}

Recall that a locally convex space $V$ is called Montel if every
bounded subset $B \subseteq V$ has compact closure. In this case, we
get the following interplay between continuity and boundedness of
functions defined on closed sets.
\begin{lemma}
    \label{lem:Montel}
    Let $V$ and $W$ be locally convex spaces such that $V$ is Montel and
    $W$ is Hausdorff. Moreover, let $f \colon C \longrightarrow W$ be a
    continuous function defined on a closed subset~$C \subseteq V$. Then
    $f$ is bounded.
\end{lemma}
\begin{proof}
    Let $B \subseteq C$ be bounded. Invoking the Montel property,
    its closure $B^\cl \subseteq C^\cl = C$ is compact. Hence, the
    same is true for its image $f(B^\cl)$ under the continuous
    mapping~$f$. As a compact set, the image $f(B^\cl)$ is a
    bounded superset of~$f(B)$ by the Hausdorff property of $W$. This
    proves the boundedness of $f(B)$, and by variation of $B$, also of
    $f$.
\end{proof}

In particular, globally defined continuous functions are bounded.

%% file: TeX/Polynomials.tex

Having reminded and convinced ourselves of these general results, we
turn towards polynomial mappings between, unless explicitly stated
otherwise, complex locally convex spaces~$V$ and $W$. This notion is
discussed comprehensively within \cite[Ch.~1]{dineen:1999a}. Given an
$n$-linear mapping $L \colon V \times \cdots \times V \longrightarrow W$,
we call
\begin{equation}
    \label{eq:PolynomialFromMultilinear}
    P
    \colon
    V \longrightarrow W, \quad
    P(v)
    \coloneqq
    L(v, \ldots, v)
\end{equation}
the polynomial induced by $P$. It is homogeneous of degree $n \in
\field{N}_0$, i.e.
\begin{equation}
    \label{eq:Homogeneity}
    P(z \cdot v)
    =
    z^n
    \cdot
    P(v)
    \qquad
    \textrm{for all }
    z \in \field{C},
    v \in V.
\end{equation}
Thus, we speak of homogeneous polynomials of degree $n$ in the
sequel, and refer to linear combinations of such functions as
polynomials. This yields a $\Z$-graded vector space, and in the case of
$W = \C$ even a $\Z$-graded algebra. Note that, in negative degrees we
simply put the zero vector space.

Remarkably, one may always reconstruct the symmetric part
\begin{equation}
    \label{eq:SymmetricPart}
    L_\symmetric
    \colon
    V^n \longrightarrow W, \quad
    L_\symmetric(v_1, \ldots, v_n)
    \coloneqq
    \frac{1}{n!}
    \sum_{\sigma \in \group{S}_n}
    L
    \bigl(
    v_{\sigma(1)}, \ldots, v_{\sigma(n)}
    \bigr)
\end{equation}
of $L$ from $P$ in an explicit manner by means of polarization. Here,
$\group{S}_n$
denotes the symmetric group in the letters $\{1,\ldots,n\}$.
\begin{proposition}[Polarization identity,
{\cite[Cor.~1.6]{dineen:1999a}}]
    \label{prop:Polarization}
    Let $V, W$ be vector spaces over a field with characteristic zero
    and let~$P \colon V \longrightarrow W$ an $n$-homogeneous polynomial
    induced by $L \colon V^n \longrightarrow W$. Then
    \begin{equation}
        \label{eq:Polarization}
        L_\symmetric(v_1, \ldots, v_n)
        =
        \frac{1}{2^n \cdot n!}
        \sum_{\epsilon_1, \ldots, \epsilon_n = \pm 1}
        \epsilon_1 \cdots \epsilon_n
        \cdot
        P
        \biggl(
        \sum_{k=1}^{n}
        \epsilon_k v_k
        \biggr)
    \end{equation}
    for $v_1, \ldots, v_n \in V$.
\end{proposition}

Thus, algebraically, symmetric multilinear mappings and
polynomials are one and the same. That is to say, $L_\symmetric$ is the
unique symmetric $n$-linear mapping inducing~$P$. In the sequel, we
denote it shorthand by $\widecheck{P}$. The continuity of $P$ and
$\widecheck{P}$ are in direct correspondence by \eqref{eq:Polarization}.
The same can be said about boundedness by the following estimate.

Recall that a subset $S \subseteq V$ of a complex vector space $V$ is
called balanced if $z \cdot v \in S$ for all $v \in S$ and $z \in \C$
with~$\abs{z} = 1$. If the set~$S$ is moreover convex, one refers to it
as absolutely convex.
\begin{corollary}[Polarization estimate,
{\cite[(1.14)]{dineen:1999a}}]
    \label{cor:PolynomialContinuityVsMultilinear}%
    Let $P \colon V \longrightarrow W$ be a homogeneous polynomial of
    degree $n \in \field{N}$ between locally convex spaces $V$ and $W$.
    Then
    \begin{equation}
        \label{eq:PolarizationEstimate}
        \seminorm{p}_{B,\seminorm{q}}(P)
        \le
        \sup_{v_1, \ldots, v_n \in B}
        \seminorm{q}
        \bigl(
        \widecheck{P}(v_1, \ldots, v_n)
        \bigr)
        \le
        \frac{n^n}{n!}
        \cdot
        \seminorm{p}_{B,\seminorm{q}}(P)
    \end{equation}
    holds for absolutely convex subsets $B \subseteq V$ and
    $\seminorm{q} \in \cs(W)$.
\end{corollary}

Next, we note an elementary continuity estimate for translations on the
space $\Pol^n_\beta(V,W)$.
\begin{proposition}[Translations of Polynomials,
{\cite[Lem.~1.10]{dineen:1999a}}]
    \label{prop:PolynomialsTranslations}
    Let $V$ and $W$ be locally convex spaces,~$B \subseteq V$ be a
    balanced subset, $\seminorm{q} \in \cs(W)$ and $P \colon
    V \longrightarrow W$ be an $n$-homogeneous polynomial for some $n \in
    \N_0$. Then
    \begin{equation}
        \label{eq:PolynomialsTranslationForward}
        \seminorm{p}_{B,\seminorm{q}}
        (P)
        \le
        \seminorm{p}_{v_0 + B,\seminorm{q}}
        (P)
        \qquad
        \textrm{for }
        v_0 \in V.
    \end{equation}
    If $B$ is moreover convex and $r > 0$ as well as $v_0 \in V$ are
    such that $r v_0 \in B$, then conversely
    \begin{equation}
        \label{eq:PolynomialsTranslationBack}
        \seminorm{p}_{B + v_0,\seminorm{q}}
        (P)
        \le
        \bigl(
        1 + 1/r
        \bigr)^n
        \cdot
        \seminorm{p}_{B,\seminorm{q}}
        (P).
    \end{equation}
\end{proposition}

Note that we have not assumed the boundedness of $B$ within
Proposition~\ref{prop:PolynomialsTranslations}. Typically, this will
lead to the degeneration of \eqref{eq:PolynomialsTranslationForward} and
\eqref{eq:PolynomialsTranslationBack} to the vacuous statement $\infty
\le \infty$. In the proof of Proposition~\ref{prop:FrechetBaire} we are
going to encounter a rather special situation, where this is not the
case.

We return to continuity: As the continuity of $\widecheck{P}$ is
equivalent to its continuity at zero, the same is true for $P$. Using
the local convexity, this may in turn be rephrased as a continuity
estimate. Indeed, an $n$-homogeneous polynomial $P$ is continuous at
zero iff for every $\seminorm{q} \in \cs(W)$ there exists some
$\seminorm{p} \in \cs(V)$ such that
\begin{equation}
    \label{eq:PolynomialContinuityEstimate}
    \seminorm{q}
    \bigl(
        P(v)
    \bigr)
    \le
    \seminorm{p}(v)^n
    \qquad
    \textrm{for }
    v \in V.
\end{equation}

Consequently, every continuous polynomial is bounded. Assuming
first countability of the domain, which in particular implies that
continuity reduces to sequential continuity, we may show the
converse. Before doing so, we note as a preliminary consideration that
both continuity and boundedness are decided within each homogeneous
degree individually.
\begin{lemma}
    \label{lem:Homogeneity}%
    Let $P \colon V \longrightarrow W$ be a polynomial between
    Hausdorff locally convex spaces with homogeneous components
    $P_n$ for $n \in \N_0$.
    \begin{lemmalist}
        \item \label{item:HomogeneityContinuity}%
        The polynomial $P$ is continuous iff all $P_n$ are.

        \item \label{item:HomogeneityBoundedness}%
        The polynomial $P$ is bounded iff all $P_n$ are.
    \end{lemmalist}
\end{lemma}
\begin{proof}
    Assuming the boundedness of all $P_n$ immediately implies
    \begin{equation*}
        \seminorm{p}_{B,\seminorm{q}}(P)
        \le
        \sum_{n=0}^\infty
        \seminorm{p}_{B,\seminorm{q}}(P_n)
        <
        \infty
    \end{equation*}
    for bounded subsets $B \subseteq V$ and $\seminorm{q} \in
    \cs(W)$, as all but finitely many $P_n$ are zero and the
    remaining seminorms are finite by assumption. Likewise, it is clear
    that the continuity of all $P_n$ implies the continuity of $P$.
    Assuming conversely that $P$ is continuous, we may use the
    vector-valued Cauchy integral formula\footnote{A proper explanation
    of this point of view can be found at the beginning of
    Section~\ref{subsec:FunctionsFrechet}.} to recover
    \begin{equation}
        P_n(v)
        =
        \frac{1}{2\pi}
        \int_{0}^{2\pi}
        \E^{-\I n t}
        \cdot
        P(\E^{\I t} \cdot v)
        \D t
        \qquad
        \textrm{for }
        v \in V
    \end{equation}
    from $P$ by virtue of its homogeneity. But then it is clear that
    $P_n$ inherits the continuity of $P$. For the analogous implication
    for boundedness, let $B \subseteq V$ be bounded, $(v_k) \subseteq B$
    and $(z_k) \subseteq \C$ a zero sequence. By boundedness of
    $P$, we then know that
    \begin{equation}
        \label{eq:PolynomialOnZeroSequence}
        P(z \cdot z_k \cdot v_k)
        =
        \sum_{n=0}^\infty
        z^n
        \cdot
        P_n(z_k \cdot v_k)
        \tag{$\heartsuit$}
    \end{equation}
    converges to zero in $W$ for every fixed $z \in \C$, and our
    goal is to show that this already implies the convergence of
    $P_n(z_k \cdot v_k)$ to zero for all $n \in \N_0$. The crucial
    point is again that $P_n = 0$ for all $n > N$ with a suitable
    $N \in \N$. Viewing \eqref{eq:PolynomialOnZeroSequence} as a
    polynomial in $z$, we may now use polynomial interpolation as
    e.g. described in the textbook
    \cite[Sec.~9.2]{humpherys.jarvis:2020a} to
    write its coefficients $P_n(z_k \cdot v_k)$ as a linear
    combination of its values $P(z \cdot z_k \cdot v_k)$, say at~$z =
    0,\ldots,N$. But then the convergence of $P(j \cdot z_k
    \cdot v_k) \rightarrow 0$ readily implies the convergence of
    each coefficient $P_n(z_k \cdot v_k) \rightarrow 0$.
\end{proof}

The second ingredient we shall need is the algebraic Taylor expansion
\begin{equation}
    \label{eq:PolynomialTaylor}
    P(v+w)
    =
    \sum_{k=0}^{n}
    \binom{n}{k}
    \widecheck{P}
    \bigl(
    \underbrace{v,\ldots,v}_{k\text{-times}},
    \underbrace{w,\ldots,w}_{(n-k)\text{-times}}
    \bigr)
    \qquad
    \textrm{for }
    v,w \in V
\end{equation}
of any polynomial $P \colon V \longrightarrow W$ of homogeneous
degree $n$, which is \cite[Lem.~1.9]{dineen:1999a}.
\begin{proposition}
    \label{prop:PolynomialsBoundedness}%
    Let $P \colon V \longrightarrow W$ be a bounded polynomial
    between Hausdorff locally convex spaces~$V$ and $W$ such that $V$ is
    first countable. Then $P$ is continuous.
\end{proposition}
\begin{proof}
    We prove that discontinuous polynomials are unbounded. As
    translations constitute homeomorphisms on locally convex spaces, we
    may assume that our polynomial is discontinuous at zero. Here we use
    that, by virtue of \eqref{eq:PolynomialTaylor}, the translated
    polynomial is still a polynomial. By Lemma~\ref{lem:Homogeneity},
    \ref{item:HomogeneityContinuity} we then get a discontinuous
    homogeneous component~$P \coloneqq P_k$. This yields a zero sequence
    $(v_n)_n \subseteq V$ such that its image $(P(v_n))_n$ is not
    converging to zero in $W$. Hence, there exists a continuous seminorm
    $\seminorm{q} \in \cs(W)$, a subsequence $(w_n)_n$ of $(v_n)_n$ and
    $r > 0$ such that
    \begin{equation*}
        \seminorm{q}
        \bigl(
        P(w_n)
        \bigr)
        \ge
        r
        >
        0
        \qquad
        \textrm{for }
        n \in \N_0.
    \end{equation*}
    By first countability, we moreover find a countable ascending
    defining system of seminorms~$\{\seminorm{p_n}\}$ for $V$. As
    $(w_n)_n$ is a zero sequence, there exist strictly ascending
    $M_n \in \N$ such that
    \begin{equation*}
        \seminorm{p}_n(w_m)
        \le
        \frac{1}{n}
        \qquad
        \textrm{for }
        m \ge M_n.
    \end{equation*}
    Consider now the subsequence $(w_{M_n})_n$ of $(v_n)_n$. By
    construction, it fulfils
    \begin{equation*}
        \seminorm{q}
        \bigl(
        P(w_{M_n})
        \bigr)
        \ge
        r
        \qquad
        \textrm{for }
        n \in \N
    \end{equation*}
    and thus
    \begin{equation*}
        \seminorm{q}
        \bigl(
        P(n \cdot w_{M_n})
        \bigr)
        =
        n^k
        \cdot
        \seminorm{q}
        \bigl(
        P(w_{M_n})
        \bigr)
        \ge
        n^k
        \cdot
        r
        \overset{n \rightarrow \infty}{\longrightarrow}
        \infty.
    \end{equation*}
    Conversely, if $\ell \in \N$ and $n \ge \ell$, then
    \begin{equation*}
        \seminorm{p}_\ell
        (n \cdot w_{M_n})
        \le
        n
        \cdot
        \seminorm{p}_n
        (w_{M_n})
        \le
        n
        \cdot
        \frac{1}{n}
        =
        1,
    \end{equation*}
    as our defining system is ascending. This covers almost all
    indices, which means
    \begin{equation*}
        \sup_{n \in \N}
        \seminorm{p}_\ell
        (n \cdot w_{M_n})
        <
        \infty.
    \end{equation*}
    Variation of $\ell \in \N$ asserts the boundedness of $(n \cdot
    w_{M_n})_n$, showing the unboundedness of~$P$. Finally, invoking
    Lemma~\ref{lem:Homogeneity}, \ref{item:HomogeneityBoundedness} this
    implies the unboundedness of the full polynomial, completing the
    proof.
\end{proof}

Beyond first countable domains,
Proposition~\ref{prop:PolynomialsBoundedness} may fail, even in the
bornological setting. Thus general polynomials behave differently than
linear mappings in this regard.
\begin{example}
    \label{ex:DiscontinuousBoundedPolynomial}%
    Consider the doubly indexed sequence spaces
    \begin{equation}
        V_n
        \coloneqq
        \bigl\{
        (a_{k\ell})_{k,\ell \in \N_0}
        \;\big|\;
        \forall_{\ell \ge n, k \in \N_0}
        \colon
        a_{k \ell}
        =
        0
        \bigr\}
        \qquad
        \textrm{for }
        n \in \N
    \end{equation}
    endowed with the locally convex topology induced by the seminorms
    \begin{equation}
        \seminorm{q}^{(n)}_{k_0}(a)
        \coloneqq
        \max_{0 \le k \le k_0}
        \max_{0 \le \ell < n}
        \abs[\big]
        {
            a_{k \ell}
        }
    \end{equation}
    indexed by $k_0 \in \N$. Leaving out the zero columns provides
    isomorphisms
    \begin{equation}
        V_n
        \cong
        \biggl(
            \prod_{m=0}^\infty \C
        \biggr)^n
    \end{equation}
    of locally convex spaces, where we endow the Cartesian products
    with the Cartesian product topology arising from the unique Hausdorff
    topological vector space topology of~$\C$. Hence, the spaces
    $V_n$ are naturally ascending Fréchet spaces with linear
    embeddings
    \begin{equation}
        V_n \hookrightarrow V_{n+1}
        \qquad
        \textrm{for }
        n \in \N.
    \end{equation}
    That is to say, $V \coloneqq \varinjlim V_n$ constitutes a
    strict LF space, which as a set is just the union over all
    $V_n$. In particular, it is not first countable by Baire's
    Theorem. A convenient and concrete defining system of
    seminorms for $V$ is given by
    \begin{equation}
        \seminorm{q}_{r,d}(a)
        \coloneqq
        \sup_{\ell \in \N_0}
        r_\ell
        \cdot
        \sup_{0 \le k \le d_\ell}
        \abs[\big]
        {
            a_{k\ell}
        },
    \end{equation}
    where we vary $r \in \Map(\N_0,(0,\infty))$ and $d \in
    \Map(\N_0,\N_0)$. Indeed, their continuity when restricted to
    each $V_n$ is straightforward. Conversely, starting with a
    continuous seminorm on $V$, the continuity estimates of its
    restrictions to $V_n$ allow for the construction of suitable
    sequences $r$ and $d$. For brevity, we write
    \begin{equation}
        J
        \coloneqq
        \Map
        \bigl(
            \N_0,(0,\infty)
        \bigr)
        \times
        \Map(\N_0,\N_0),
    \end{equation}
    which becomes a directed set with respect to the pointwise order in
    both arguments. Consider the quadratic polynomial
    \begin{equation}
        \label{eq:DiscontinuousPolynomial}
        P
        \colon
        V \longrightarrow \C, \quad
        P(a)
        \coloneqq
        \sum_{n=0}^{\infty}
        a_{n0} \cdot a_{0n},
    \end{equation}
    where the series of course terminates for every $a \in V$ and is
    thus only formally infinite. Restricting $P$ to $V_N$ yields
    continuous polynomials $P_N$, as
    \begin{equation}
        \abs[\big]
        {P_N(a)}
        \le
        \sum_{n=0}^{N-1}
        \abs[\big]
        {a_{n0}}
        \cdot
        \abs[\big]
        {a_{0n}}
        \le
        N
        \cdot
        \seminorm{q}^{(N)}_{N}(a)
        \cdot
        \seminorm{q}^{(N)}_0(a).
    \end{equation}
    In particular, $P_N$ is bounded. Now, if $B \subseteq V$
    is bounded, then we find some $N \in \N_0$ such that $B
    \subseteq V_{N}$ is a bounded subset by regularity of
    countable strict inductive limits, see for instance \cite[Sec.~4.6
    Theorem~2]{jarchow:1981a}. But then $P(B) = P_N(B)$
    is bounded by continuity of~$P_N$. Variation of $B$ establishes
    the boundedness of $P$. In passing, we note that this also
    means that
    \begin{equation}
        P_N \rightarrow P
        \qquad
        \textrm{uniformly on bounded subsets},
    \end{equation}
    where we identify $P_N$ with the $N$-th partial sum of the
    series \eqref{eq:DiscontinuousPolynomial}. We shall return to this
    observation in a moment, after establishing the discontinuity of
    $P$. To this end, we define a net $(a_{r,d})_{(r,d) \in J} \subseteq
    V$ of sequences $a_{r,d}$ with entries
    \begin{equation}
        a_{r,d}(0,0)
        \coloneqq
        \frac{1}{r_0},
        \quad
        a_{r,d}(0,d_0+1)
        \coloneqq
        \frac{1}{r_{d_0+1}}
        \qquad \textrm{as well as} \qquad
        a_{r,d}(d_0+1,0)
        \coloneqq
        r_{d_0 + 1}
    \end{equation}
    and zero otherwise. Then, on the one hand,
    \begin{equation}
        \label{eq:ToyModelNetInBalls}
        \seminorm{q}_{r,d}(a_{r,d})
        =
        \sup_{\ell \in \N_0}
        r_\ell
        \cdot
        \sup_{0 \le k \le d_\ell}
        \abs[\big]
        {
            a_{r,d}(k,\ell)
        }
        =
        1,
    \end{equation}
    as the seminorm only sees the non-zero entries at $(0,0)$ and
    $(0,d_0+1)$, but not the one at $(d_0+1,0)$. On the other hand, we
    have
    \begin{equation}
        \label{eq:ToyModelNetPolynomial}
        P(a_{r,d})
        =
        \frac{1}{r_0^2}
        +
        \frac{1}{r_{d_0+1}}
        \cdot
        r_{d_0+1}
        \ge
        1,
    \end{equation}
    as the first term is always positive. Taking a step back, we
    have established the discontinuity of $P$. It is
    straightforward to prove that $P$ is nevertheless both
    sequentially and separately continuous by using the regularity
    of the inductive limit and its universal property,
    respectively.
\end{example}

Returning to functional analytic considerations, we write
$\Pol^n_\beta(V,W)$ for the space of continuous polynomials between
locally convex spaces $V$ and $W$ endowed with the topology of
uniform convergence on bounded subsets. The following statement can
be found in \cite[Prop.~1.30]{dineen:1981a}, albeit without proof.
Moreover, it does not assume first countability, which by
Example~\ref{ex:DiscontinuousBoundedPolynomial} however makes the
claim erroneous.
\begin{proposition}
    \label{prop:CompletenessHomogeneousPolynomials}%
    Let $V$ be a first countable locally convex space and $W$
    be (sequentially) complete Hausdorff. Then $\Pol^n_\beta(V,W)$
    is (sequentially) complete Hausdorff for all $n \in \N_0$.
\end{proposition}
\begin{proof}
    The case $n=0$ corresponds just to the assumptions on $W$. Thus
    let $n \in \N$. Using once again that continuous polynomials
    are bounded, we get the embedding
    \begin{equation*}
        \Pol^n_\beta(V,W)
        \subseteq
        \Bounded(V,W).
    \end{equation*}
    As the space of bounded mappings $\Bounded(V,W)$ is complete
    Hausdorff by Proposition~\ref{prop:BoundedCompleteness}, this means
    that $\Pol^n_\beta(V,W)$ is Hausdorff itself and that it
    suffices to establish the closedness of $\Pol^n_\beta(V,W)$
    within $\Bounded(V,W)$. To this end, let
    \begin{equation*}
        (P_\alpha)_{\alpha \in J}
        \subseteq
        \Pol^n_\beta(V,W)
    \end{equation*}
    be a convergent net with limit $P \in \Bounded(V,W)$. As the
    pointwise convergence of $(P_\alpha)_{\alpha \in J}$
    corresponds to pointwise convergence of the corresponding net
    $(\widecheck{P}_\alpha)$ of $n$-linear mappings by
    Proposition~\ref{prop:Polarization}, it is clear that $P$ is an
    $n$-homogeneous polynomial itself. But then invoking
    Proposition~\ref{prop:PolynomialsBoundedness} yields its
    continuity.
\end{proof}

%% file: TeX/Frechet.tex
The final type of bounded functions we are interested in are the bounded
Fréchet holomorphic ones. Let us take a moment to clarify what we
mean by this. Our presentation is adapted from the textbook
\cite[Ch.~3]{dineen:1999a}, which contains a variety of additional
material, but mostly limits itself to the simpler situation of globally
defined mappings.

Let $V$ and $W$ be again Hausdorff locally convex spaces and $U
\subseteq V$ an open connected subset, which we refer to as a
domain in the sequel. We call a mapping~$f \colon U \longrightarrow W$
\emph{Fréchet holomorphic} and write $f \in
\Holomorphic(U,W)$ if it is continuous and the functions of one complex
variable $z$ given by
\begin{equation}
    \label{eq:Gateaux}
    z
    \mapsto
    w'
    \bigl(
        f(v_0 + zv)
    \bigr)
\end{equation}
are holomorphic on their domains for all $v_0 \in U$, $v \in V$ and
$w' \in W'$. Remarkably, already the latter condition, which is
known as G\^ateaux holomorphy in the literature, ensures the existence
of pointwisely convergent Taylor expansions
\begin{equation}
    \label{eq:TaylorSeries}
    f(v_0 + v)
    =
    \sum_{n=0}^{\infty}
    P_{n,v_0}(v)
\end{equation}
with unique $n$-homogeneous polynomials $P_n \colon V \longrightarrow
\widehat{W}$, explicitly given by the Cauchy integrals
\begin{equation}
    \label{eq:TaylorPolynomials}
    P_{n, v_0}(v)
    =
    \frac{1}{2 \pi \I r}
    \int_{r \mathbb{S}^1}
    \frac{f(v_0 + zv)}{z^{n+1}}
    \D z.
\end{equation}
Here, we denote the completion of $W$ by $\widehat{W}$, and it would
in principle suffice to have sequential completeness to ensure the
existence of the vector-valued Riemann integrals over the contour
given by traversing the suitably scaled\footnote{Given $v_0 \in U$
and $v \in V$, one has to choose $r > 0$ such that $v_0 + zv \in U$ for
all $z \in \mathbb{S}^1$. The resulting integral does not depend on the
choice of $r$.} unit circle $\mathbb{S}^1 \subseteq \C$ once in positive
direction. In the sequel, we denote the subspace of G\^ateaux
holomorphic functions by adding a subscript $G$, e.g. by writing
$\Holomorphic_G(U,W)$.

A crucial detail is then that like in finite dimensions, the Taylor
expansion \eqref{eq:TaylorSeries} converges for all $v \in U$ such that
the line segment
\begin{equation}
    [v_0,v]
    \coloneqq
    \bigl\{
        (1-t)v_0 + tv
        \colon
        t \in [0,1]
    \bigr\}
    \subseteq
    U,
\end{equation}
regardless of the finer properties of the function $f$. For later use,
we introduce the more precise notation
\begin{equation}
    \label{eq:Polydisks}
    \mathfrak{B}(U,v_0)
    \coloneqq
    \bigl\{
        B \subseteq V
        \colon
        B \;
        \textrm{is an absolutely convex bounded set with} \;
        v_0 + B
        \subseteq
        U
    \bigr\}
\end{equation}
for any open $U \subseteq V$ and $v_0 \in U$. Indeed, if $B \in
\Bdd(U,v_0)$ is non-empty, then $B$ is starlike with starpoint zero by
absolute convexity, and thus $v_0 + B$ is starlike with starpoint~$v_0$.
Note that there is always a wealth of absolutely convex
bounded sets, as passing to absolutely convex hulls preserves
boundedness. We think of the collection \eqref{eq:Polydisks} as a
locally convex adaptation of polydisks.

Returning to \eqref{eq:TaylorPolynomials} it is moreover immediate that
the continuity of $f$ implies the continuity of each $P_{n,v_0}$, i.e.
$P_{n,v_0} \in \Pol^n_\beta(V,W)$. Our next goal is to establish the
converse for polynomials defined on first countable spaces, which is at
the heart of Theorem~\ref{thm:CompletenessBoundedFrechet}. A first
ingredient is the following locally convex incarnation of the Cauchy
estimates, which may readily be inferred from
\eqref{eq:TaylorPolynomials}.
\begin{proposition}[Locally convex Cauchy estimates,
    {\cite[(3.12)]{dineen:1999a}}]
    \label{prop:CauchyEstimatesLocallyConvex}
    Let $V$ and $W$ be Hausdorff locally convex spaces, $U \subseteq V$
    be open and $f \in \Holomorphic_G(U,W)$ with Taylor
    expansion~\eqref{eq:TaylorSeries} around $v_0 \in U$. Then, for
    every balanced
    subset $B \subseteq V$, we have
    \begin{equation}
        \label{eq:CauchyEstimatesLocallyConvex}
        \seminorm{p}_{B,\seminorm{q}}(P_{n,v_0})
        \le
        r^{-n}
        \cdot
        \seminorm{p}_{v_0 + rB,\seminorm{q}}(f)
    \end{equation}
    for $\seminorm{q} \in \cs(W)$, $n \in \N_0$ and $r >0$ with $v_0
    + rB \subseteq U$.
\end{proposition}

\begin{remark}
    \label{rem:DegenerationOfInequality}%
    One has to be somewhat mindful when applying our formulation of the
    locally convex Cauchy estimates: for general balanced sets $B
    \subseteq V$ and G\^ateaux holomorphic $f \in \Holomorphic_G(U,W)$,
    the seminorms $\seminorm{p}_{v_0 + rB,\seminorm{q}}(f)$ appearing as
    upper bounds within \eqref{eq:CauchyEstimatesLocallyConvex} have no
    reason to be finite for any $r > 0$ with $v_0 + rB \subseteq U$. In
    this case, the statement simply becomes vacuous. In the sequel, we
    will mostly be interested in bounded subsets~$B$ and bounded
    functions $f$, in which case the seminorms will always be finite. In
    Example~\ref{ex:HolomorphicUnbounded} we will see that boundedness
    is an additional property that does not automatically follow from
    Fréchet holomorphy.
\end{remark}

The locally convex Cauchy estimates facilitate several continuity
statements regarding the passage from holomorphic functions and
expansion points towards Taylor polynomials and their values. A first
instance is the following:
\begin{corollary}
    \label{cor:GateauxTaylorContinuity}
    Let $V,W$ be locally convex spaces such that $V$ is Hausdorff.
    Moreover, let $U \subseteq V$ be an absolutely convex domain, $v_0
    \in U$ and $n \in \N_0$. Then the linear projection
    \begin{equation}
        \label{eq:TaylorPolynomialProjection}
        P_{n,v_0}
        \colon
        \Holomorphic_G(U,W)
        \longrightarrow
        \Holomorphic_G(V,W)
    \end{equation}
    is continuous with respect to the topology of locally uniform
    convergence on finite dimensional subspaces.
\end{corollary}
\begin{proof}
    Let $F \subseteq V$ be a finite dimensional subspace. Without loss
    of generality, we may assume $v_0 \in F$. Then $F \cap U \subseteq
    F$ is an open neighbourhood of $v_0$ and thus contains a compact
    ball $K_0 \coloneqq \Ball_r(v_0)^\cl \subseteq F$ with respect to
    some auxiliary norm. By uniqueness of the locally convex Hausdorff
    topology on $F$, the ball $K_0$ is compact also with respect to the
    subspace topology induced by $V$. In particular, it is bounded by
    the Hausdorff property of $V$. The locally convex Cauchy estimates
    \eqref{eq:CauchyEstimatesLocallyConvex} then yield
    \begin{equation*}
        \sup_{v \in \Ball_r(0)}
        \seminorm{q}
        \bigl(
        P_{n,v_0}(v)
        \bigr)
        \le
        \sup_{v \in K_0}
        \seminorm{q}
        \bigl(
        f(v)
        \bigr)
        \qquad
        \textrm{for }
        n \in \N_0.
    \end{equation*}
    Note that the warning from Remark~\ref{rem:DegenerationOfInequality}
    does not apply here, as the restriction $f \at{F \cap U}$ is
    continuous and as such bounded on the compact set $K_0$ by Hartogs'
    Theorem on separate holomorphicity as it can e.g. be found in
    \cite[Thm.~1.2.5]{krantz:2001a}. If now $K \subseteq F$ is an
    arbitrary compact set, then $K \subseteq \Ball_R(0)$ for some $R >
    0$ and thus the homogeneity \eqref{eq:Homogeneity} and what we have
    already proved yield
    \begin{equation*}
        \sup_{v \in K}
        \seminorm{q}
        \bigl(
        P_{n,v_0}(v)
        \bigr)
        \le
        \frac{R^n}{r^n}
        \sup_{v \in \Ball_r(v_0)}
        \seminorm{q}
        \bigl(
        f(v)
        \bigr).
    \end{equation*}
    Variation of $F$ completes the proof.
\end{proof}

The same ideas also show that the Taylor series \eqref{eq:TaylorSeries}
converges locally uniformly on finite dimensional subspaces. A nice
consequence of this is that the continuity of Taylor polynomials is
consistent across the domain, i.e. either $P_{n,v}$ is continuous for
either none or all $v \in U$. To show this, we also need the algebraic
Taylor expansion \eqref{eq:PolynomialTaylor}.
\begin{lemma}[Consistency Lemma]
    \label{lem:TaylorPolynomialsContinuityPropagation}%
    Let $V,W$ be locally convex spaces such that $W$ is sequentially
    complete Hausdorff, $U \subseteq V$ be a domain and $f \in
    \Holomorphic_G(U,W)$. Assume furthermore that there is some point
    $v_0 \in U$ such that the Taylor polynomials $P_{n,v_0}$ are
    continuous for all $n \in \N_0$. Then the Taylor polynomials
    $P_{n,v}$ are continuous for all $v \in U$ and $n \in \N_0$.
\end{lemma}
\begin{proof}
    Let $U_0 \subseteq V$ be an absolutely convex zero neighbourhood
    such that $U_0 + v_0 \subseteq U$, so that we may describe $f$ by
    its Taylor expansion around $v_0$ on all of $U_0 + v_0$. Let
    now~$v \in U_0$ and $w \in V$. Invoking
    Corollary~\ref{cor:GateauxTaylorContinuity},
    and indicating the function dependence with square brackets, we have
    \begin{align*}
        P_{n,v}
        \bigl[f(v_0 + \argument)\bigr]
        \at[\Big]{w}
        &=
        P_{n,v}
        \Biggl[
        \sum_{m=0}^{\infty}
        P_{m,v_0}
        \Biggr]
        \at[\bigg]{w} \\
        &=
        \sum_{m=0}^{\infty}
        P_{n,v}
        [P_{m,v_0}]
        \at[\Big]{w} \\
        &=
        \sum_{m=0}^n
        \binom{n}{m}
        \widecheck{P}_{m,v_0}
        \bigl(
            \underbrace{w,\ldots,w}_{n\text{-times}},
            \underbrace{v,\ldots,v}_{(m-n)\text{-times}}
        \bigr),
    \end{align*}
    where we have also used \eqref{eq:PolynomialTaylor} and that the
    Taylor series converges locally uniformly on intersections of $U$
    with finite dimensional subspaces to interchange the Cauchy
    integrals~\eqref{eq:TaylorPolynomials} with the series. By the
    assumed continuity of the $P_{m,v_0}$, each of the Taylor
    polynomials $P_{n,v}[f(v_0 + \argument)]$ is thus continuous itself
    as a finite sum of continuous functions. Now,
    \begin{equation*}
        \sum_{n=0}^{\infty}
        P_{n,v}
        \bigl[
            f(v_0 + \argument)
        \bigr]
        \at[\Big]{w}
        =
        f(v + v_0 + w)
        =
        \sum_{n=0}^{\infty}
        P_{n,v}[f]
        \at[\Big]{v_0 + w}
    \end{equation*}
    implies
    \begin{equation}
        \label{eq:TaylorExpansionVsTranslation}
        P_{n,v}
        \bigl[
            f(v_0 + \argument)
        \bigr]
        \at[\Big]{w}
        =
        P_{n,v}[f]
        \at[\Big]{v_0 + w}
        \qquad
        \textrm{for }
        n \in \N_0
    \end{equation}
    by uniqueness of the Taylor polynomials. Hence, the continuity of
    $P_{n,v}$ follows as desired. If now $v \in U$ is arbitrary, we use
    the polygonal convexity of $U$ to find points
    $v_1,\ldots,v_n \in U$ such that the line segments $[v_k,v_{k+1}]
    \subseteq U$ for $k=0,\ldots,n-1$. Then, the claim follows by
    inductive application of what we have already shown.
\end{proof}

The polygonal convexity of domains within locally convex spaces readily
follows from covering the trace of a path connecting $v_0$ to $v$ within
$U$ by convex open neighbourhoods of its constitutent points contained
within $U$ and using compactness. The relation
\eqref{eq:TaylorExpansionVsTranslation} states that translations commute
with differentiation.

Before moving on, we note that locally uniform convergence on
finite dimensional subspaces preserves G\^ateaux holomorphy, a fact we
are going to need later. A generating system of seminorms for this
locally convex topology is given by
\begin{equation}
    \label{eq:GateauxSeminorms}
    \Holomorphic_G(U,W)
    \ni
    f
    \quad \mapsto \quad
    \max_{v \in K}
    \seminorm{q}
    \bigl(
    f(v)
    \bigr),
\end{equation}
where we vary over compact sets $K \subseteq F \cap U$ contained in
finite dimensional subspaces~$F$ of $V$ intersected with the open set
$U$ and continuous seminorms $\seminorm{q} \in \cs(W)$.
\begin{proposition}[Completeness of $\Holomorphic_G$]
    \label{prop:GateauxTopology}
    Let $V,W$ be Hausdorff locally convex spaces such that $W$ is
    (sequentially) complete, and let~$U \subseteq V$ be open. Then the
    space $\Holomorphic_G(U,W)$ is (sequentially) complete with respect
    to the topology of locally uniform convergence on finite dimensional
    subspaces.
\end{proposition}
\begin{proof}
    Let $(f_\alpha)_{\alpha \in J} \subseteq \Holomorphic_G(U,W)$ be a
    Cauchy net. As $W$ is complete and the topology of locally uniform
    convergence on finite dimensional subspaces is finer than the
    topology of pointwise convergence, we get pointwise convergence to
    the pointwise limit
    \begin{equation*}
        f
        \colon
        U \longrightarrow W, \quad
        f(v)
        \coloneqq
        \lim_{\alpha \in J}
        f_\alpha(v).
    \end{equation*}
    Fix $v_0 \in U$, $v \in V$ and $w' \in W'$. Consider the
    auxiliary net $(g_\alpha)_{\alpha \in J}$ given by
    \begin{equation*}
        g_\alpha
        \colon
        U_v \longrightarrow \C, \quad
        g_\alpha(v)
        \coloneqq
        w'
        \bigl(
            f_\alpha(v_0+zv)
        \bigr),
    \end{equation*}
    where
    \begin{equation*}
        U_v
        \coloneqq
        \bigl\{
            z \in \C
            \colon
            v_0 + zv
            \in
            U
        \bigr\}.
    \end{equation*}
    By assumption, $(g_\alpha)_{\alpha \in J}$ is Cauchy in
    $\Holomorphic(U_v)$ and thus locally uniformly convergent to some $g
    \in \Holomorphic(U_v)$. By continuity of $w'$, we get
    \begin{equation*}
        g(v)
        =
        w'
        \bigl(
        f(v_0+zv)
        \bigr)
        \qquad
        \textrm{for all }
        v \in U_v.
    \end{equation*}
    Variation of the data proves $f \in \Holomorphic_G(U,W)$ and
    $f_\alpha \rightarrow f$ locally uniformly on finite dimensional
    subspaces.
\end{proof}

With these preliminaries, we may now establish the following sufficient
condition for Fréchet holomorphy. Recall that a locally convex space
is called Baire if all countable unions of closed sets with empty
interior have empty interior themselves. Saxon \cite{saxon:1974a}
characterized Baire topological vector spaces by a property reminiscent
of barelledness.
\begin{proposition}
    \label{prop:FrechetBaire}%
    Let $V,W$ be locally convex spaces such that $V$ is Baire and $W$ is
    complete. Furthermore let $U \subseteq V$ be a domain and $f \in
    \Holomorphic_G(U,W)$ with continuous Taylor polynomials of all
    orders at one point in $U$. Then $f \in \Holomorphic(U,W)$.
\end{proposition}
\begin{proof}
    We have to prove continuity of $f$, which is equivalent to
    continuity of the compositions $\pi_\seminorm{q} \circ f \colon U
    \longrightarrow W_\seminorm{q}$ for all $\seminorm{q} \in \cs(W)$,
    where
    \begin{equation*}
        \pi_\seminorm{q}
        \colon W
        \longrightarrow
        W_\seminorm{q}
    \end{equation*}
    denotes the local Banach space at $\seminorm{q}$. That is to say,
    $W_\seminorm{q}$ is the completion of the
    quotient~$W/\ker\seminorm{q}$ with respect to the norm
    $\norm{[w]}_{\seminorm{q}} \coloneqq \seminorm{q}(w)$ for $w \in W$,
    where
    \begin{equation*}
        \ker \seminorm{q}
        \coloneqq
        \bigl\{
            w \in W
            \colon
            \seminorm{q}(w)
            =
            0
        \bigr\}.
    \end{equation*}
    We fix $v_0 \in U$. Invoking the Consistency
    Lemma~\ref{lem:TaylorPolynomialsContinuityPropagation},
    we get the continuity of
    \begin{equation*}
        P_n
        \coloneqq
        P_{n,v_0}
        \bigl[
            \pi_\seminorm{q}
            \circ
            f
        \bigr]
        \qquad
        \textrm{for }
        n \in \N_0.
    \end{equation*}
    Here, we have used that the quotient projection $\pi_\seminorm{q}$
    is continuous and preserves Gâteaux holomorphy by its linearity. The
    strategy is now to prove that $\pi_\seminorm{q} \circ f$ is
    continuous on some open neighbourhood of $v_0$.

    By translating and shrinking $U$ if necessary, we may assume
    that~$v_0 = 0$ and that $U$ is absolutely convex. Consider
    \begin{equation*}
        U_n
        \coloneqq
        \bigcap_{m=0}^\infty
        \bigl\{
        v \in U
        \colon
        \seminorm{q}
        \bigl(
        P_m(v)
        \bigr)
        \le
        n
        \bigr\}
        \qquad
        \textrm{for }
        n \in \N,
    \end{equation*}
    which is a closed subset of $U$ as an intersection of such, where we
    use the continuity of the Taylor polynomials $P_m$. By convergence
    of the Taylor series \eqref{eq:TaylorSeries} for fixed $v \in U$, we
    know that $(P_m(v))_m \subseteq W$ is a zero sequence
    and thus $\bigcup_{n \in \N} U_n = U$.

    By the Baire property, there thus exists an index $N \in \N_0$ and
    $v_0' \in U_N$ such that $v_0'$ is an interior point of $U_N$. Let
    $U_0 \subseteq V$ be an absolutely convex zero neighbourhood such
    that
    \begin{equation*}
        v_0' + U_0 \subseteq U_N.
    \end{equation*}
    Using \eqref{eq:Homogeneity} and
    \eqref{eq:PolynomialsTranslationForward}, we estimate
    \begin{equation*}
        \sum_{n=0}^\infty
        \seminorm{p}_{U_0/2, \seminorm{q}}(P_n)
        =
        \sum_{n=0}^\infty
        2^{-n}
        \cdot
        \seminorm{p}_{U_0, \seminorm{q}}(P_n)
        \le
        \sum_{n=0}^\infty
        2^{-n}
        \cdot
        \seminorm{p}_{v_0' + U_0, \seminorm{q}}(P_n)
        \le
        2N.
    \end{equation*}
    Hence, the Taylor series \eqref{eq:TaylorSeries} converges
    uniformly on the open neighbourhood $U_0/2$ of~$v_0 = 0$ and
    consequently $\pi_\seminorm{q} \circ f$ is continuous on $U_0/2$. We
    have shown
    \begin{equation*}
        \pi_\seminorm{q}
        \circ f
        \in
        \Holomorphic
        \bigl(
        U_0/2,W_{\seminorm{q}}
        \bigr).
    \end{equation*}
    Variation of $v_0 \in U$ and then $\seminorm{q} \in \cs(W)$
    completes the proof.
\end{proof}

\begin{example}[An unbounded holomorphic function]
    \label{ex:HolomorphicUnbounded}
    Let $V \coloneqq \ell^1$ be the space of absolutely
    summable complex sequences indexed by $\N_0$. We write $e_n \in
    \ell^1$ for the sequences with
    \begin{equation}
        e_n(k)
        \coloneqq
        \delta_{n,k}
        \qquad
        \textrm{for }
        n,k \in \N_0.
    \end{equation}
    The dual vectors
    \begin{equation}
        e_n'
        \colon
        \ell^1 \longrightarrow \C, \quad
        e_n'(a)
        \coloneqq
        a_n
    \end{equation}
    are continuous linear functionals with operator norm
    $\norm{e_n'} = 1$ for $n \in \N_0$. Define
    \begin{equation}
        \label{eq:HolomorphicUnbounded}
        f \colon \ell^1 \longrightarrow \C, \quad
        f(a)
        \coloneqq
        \sum_{n=0}^\infty
        e_n'(a)^n
        =
        \sum_{n=0}^{\infty}
        a_n^n.
    \end{equation}
    This is indeed well defined, as given $a \in \ell^1$, there exists
    an index $N \in \N_0$ with $\abs{a_n} \le 1$ for $n \ge N$.
    Consequently,
    \begin{equation}
        \abs[\big]
        {f(a)}
        \le
        \sum_{n=0}^\infty
        \abs{a_n}^n
        \le
        \sum_{n=0}^{N-1}
        \abs{a_n}^n
        +
        \sum_{n=N}^\infty
        \abs{a_n}
        \le
        \sum_{n=1}^{N-1}
        \abs{a_n}^n
        +
        \norm{a}_1
        <
        \infty.
    \end{equation}
    By construction, $f$ is Gâteaux holomorphic with Taylor polynomials
    at the origin given by $P_{n,0} = (e_n')^n$, each of which is
    $n$-homogeneous and continuous. By
    Proposition~\ref{prop:FrechetBaire}, this implies the Fréchet
    holomorphy of $f$ on all of $\ell^1$.

    However, $f$ is unbounded on every ball with radius $R > 1$.
    Indeed, if $1 < r < R$, then
    \begin{equation}
        \abs[\big]
        {f(r e_n)}
        =
        r^n
        \qquad
        \textrm{for }
        n \in \N_0,
    \end{equation}
    which is unbounded for $n \rightarrow \infty$. Conversely, the
    function $f$ is bounded by one on the closed unit ball. In passing,
    we note that the series within \eqref{eq:HolomorphicUnbounded}
    converges uniformly on any ball of radius strictly less than one,
    but on no larger ball.
\end{example}

Hence, boundedness typically singles out a proper subspace
\begin{equation}
    \Holomorphic_\beta(U,W)
    \coloneqq
    \bigl\{
    f \in \Holomorphic(U,W)
    \colon
    f
    \; \textrm{is bounded}
    \bigr\}
    =
    \Holomorphic(U,W)
    \cap
    \Bounded(U,W)
\end{equation}
of the space of all Fréchet holomorphic functions. In the sequel, we
endow it with the subspace topology inherited from $\Bounded(U,W)$. By
Lemma~\ref{lem:Montel}, the following is immediate.
\begin{corollary}
    Let $V,W$ be locally convex spaces such that $V$ is Montel and $W$
    is Hausdorff. Moreover, let $U \subseteq V$ be a domain. Then
    \begin{equation}
        \bigl\{
        f \in \Holomorphic(U,V)
        \colon
        f \textrm{ admits a continuous extension to $U^\cl$}
        \bigr\}
        \subseteq
        \Holomorphic_\beta(U,V).
    \end{equation}
    In particular,
    \begin{equation}
        \Holomorphic_\beta(V,W)
        =
        \Holomorphic(V,W).
    \end{equation}
\end{corollary}

We are now in a position to establish the main result of this section.
\begin{theorem}[Completeness of $\Holomorphic_\beta$]
    \label{thm:CompletenessBoundedFrechet}%
    Let $V$ and $W$ be Hausdorff locally convex spaces such that $V$ is
    first countable Baire and $W$ is (sequentially) complete. Then
    $\Holomorphic_\beta(U,W)$ is (sequentially) complete Hausdorff for
    any domain $U \subseteq V$.
\end{theorem}
\begin{proof}
    As in the proof of
    Proposition~\ref{prop:CompletenessHomogeneousPolynomials}, we view
    $\Holomorphic_\beta(U,W)$ as a subspace of $\Bounded(U,W)$.
    Consequently, the Hausdorff property is obvious and it suffices to
    prove the closedness of the
    subspace $\Holomorphic_\beta(U,W)$ within $\Bounded(U,W)$ by its
    completeness, which we have established in
    Proposition~\ref{prop:BoundedCompleteness}.

    To this end, let $(f_\alpha)_{\alpha \in J} \subseteq
    \Holomorphic_\beta(U,W)$ be a convergent net
    with limit $f \in \Bounded(U,W)$. Then~$f$ is Gâteaux holomorphic by
    virtue of Proposition~\ref{prop:GateauxTopology}. It remains to
    establish the continuity of $f$. Fix $v_0 \in U$. By Fréchet
    holomorphy of $f_\alpha$, each of the Taylor polynomials
    \begin{equation*}
        P_{n,v_0}[f_\alpha]
        \colon
        V
        \longrightarrow
        W
    \end{equation*}
    is continuous. Invoking the locally convex Cauchy estimates
    \eqref{eq:CauchyEstimatesLocallyConvex} and
    the linearity of~\eqref{eq:TaylorPolynomialProjection}, we get
    \begin{equation*}
        \seminorm{p}_{B,\seminorm{q}}
        \bigl(
            P_{n,v_0}[f]
            -
            P_{n,v_0}[f_\alpha]
        \bigr)
        \le
        r^{-n}
        \cdot
        \seminorm{p}_{v_0 + rB,\seminorm{q}}
        \bigl(
            f - f_\alpha
        \bigr)
    \end{equation*}
    for $\alpha \in J$, $n \in \N_0$ and $r > 0$ with $v_0 + rB
    \subseteq U$. By convergence of $(f_\alpha)$ towards~$f$ within the
    space $\Bounded(U,W)$, this implies the uniform convergence of
    $P_{n,v_0}[f_\alpha]$ towards $P_{n,v_0}[f]$ on bounded sets for all
    $n \in \N_0$. Invoking now the completeness of $\Pol_\beta^n(V,W)$
    for~$n \in \N_0$ from
    Proposition~\ref{prop:CompletenessHomogeneousPolynomials} this, in
    turn, implies that $P_{n,v_0}[f]$ is continuous for all~$n \in \N_0$.
    As $V$ is Baire by assumption, this is sufficient to ensure the
    continuity of $f$ by Proposition~\ref{prop:FrechetBaire}. That is to
    say, $f \in \Holomorphic(U,W)$ and we have completed the proof.
\end{proof}

In particular, $\Holomorphic_\beta(U,W)$ is always complete in the
setting of Fréchet spaces. As continuous polynomials are Fr\'{e}chet
holomorphic, Example~\ref{ex:DiscontinuousBoundedPolynomial} shows that
Theorem~\ref{thm:CompletenessBoundedFrechet} may fail beyond first
countable domains. Another remarkable consequence of our estimates is
that Taylor series of bounded Fréchet holomorphic functions converge
uniformly on bounded subsets.
\begin{corollary}
    Let $V$ and $W$ be Hausdorff locally convex spaces such that $W$ is
    sequentially complete. Moreover, let $U_0 \subseteq V$ be a balanced
    neighbourhood of zero, $v_0 \in V$ and consider a bounded Fréchet
    holomorphic $f \in \Holomorphic_\beta(v_0 + U_0,W)$. Then the Taylor
    series
    \begin{equation}
        \label{eq:TaylorOnBounded}
        \sum_{n=0}^{\infty}
        P_{n,v_0}
    \end{equation}
    converges absolutely towards $f(v_0 + \argument)$ in the space
    $\Holomorphic_\beta(U_0,W)$.
\end{corollary}
\begin{proof}
    Fix a continuous seminorm $\seminorm{q} \in \cs(W)$ and a bounded set
    \begin{equation*}
        B
        \in
        \mathfrak{B}(U_0, 0)
        =
        \mathfrak{B}
        \bigl(
        v_0
        +
        U_0,v_0
        \bigr).
    \end{equation*}
    By pointwise convergence of~\eqref{eq:TaylorSeries}, we have
    \begin{equation*}
        \seminorm{q}
        \biggl(
        f(v_0 + v)
        -
        \sum_{n=0}^{N-1}
        P_{n,v_0}(v)
        \biggr)
        =
        \seminorm{q}
        \biggl(
        \sum_{n=N}^{\infty}
        P_{n,v_0}(v)
        \biggr)
        \le
        \sum_{n=N}^{\infty}
        \seminorm{q}
        \bigl(
        P_{n,v_0}(v)
        \bigr)
    \end{equation*}
    for $v \in B$ and $N \in \N$, where the right-hand side may be
    infinite a priori. Taking suprema over $B$ and invoking the locally
    convex Cauchy estimates \eqref{eq:CauchyEstimatesLocallyConvex}
    nevertheless leads to
    \begin{equation*}
        \seminorm{p}_{B,\seminorm{q}}
        \biggl(
        f(v_0 + \argument)
        -
        \sum_{n=0}^{N-1}
        P_{n,v_0}
        \biggr)
        \le
        \sum_{n=N}^{\infty}
        \seminorm{p}_{B,\seminorm{q}}
        \bigl(
        P_{n,v_0}
        \bigr)
        \le
        \seminorm{p}_{v_0 + rB, \seminorm{q}}(f)
        \cdot
        \frac{r^{N}}{1-r}
    \end{equation*}
    for any $0 < r < 1$ and $N \in \N$. Thus sending $N \rightarrow
    \infty$ proves the desired convergence.
\end{proof}

In the proof, we saw that the series \eqref{eq:TaylorOnBounded} is
absolutely convergent. As indicated in the end of
\cite[Sec.~3.1]{dineen:1999a}, there is an
alternative defining set of seminorms for~$\Holomorphic_\beta(V,W)$
based on this observation, which will ultimately bridge the gap towards
the $R$-topologies in Section~\ref{sec:TensorsAsPolynomials}. We
generalize this to functions defined on
domains $U \subseteq V$. The locally convex Cauchy
estimates~\eqref{eq:CauchyEstimatesLocallyConvex} facilitate the
following, where we use the polydisks from \eqref{eq:Polydisks}.
\begin{proposition}
    \label{prop:BetaAlternativeSeminorms}
    Let $V, W$ be Hausdorff locally convex spaces, $U \subseteq V$ be a
    domain as well as $f \in \Holomorphic_G(U,W)$ be G\^ateaux
    holomorphic with Taylor polynomials $P_{n,v_0}$ at $v_0 \in U$.
    \begin{propositionlist}
        \item \label{item:AlternativeBounded}
        The function $f$ is bounded if and only if
        \begin{equation}
            \label{eq:BoundedAlternativeSeminorm}
            \seminorm{r}_{B,\seminorm{q},v_0}(f)
            \coloneqq
            \sum_{n=0}^{\infty}
            \seminorm{p}_{B, \seminorm{q}}
            \bigl(
            P_{n,v_0}
            \bigr)
            <
            \infty
        \end{equation}
        for all $v_0 \in U$, $B \in \Bdd(U,v_0)$ and $\seminorm{q} \in
        \cs(W)$.

        \item \label{item:AlternativeBoundedAreSeminorms}
        The functions $\seminorm{r}_{B, \seminorm{q}, v_0}$ are
        seminorms on $\Holomorphic_\beta(U,W)$ for all vectors $v_0 \in
        U$, bounded sets $B \in \Bdd(U,v_0)$ and continuous seminorms
        $\seminorm{q} \in \cs(W)$.
        \item \label{item:BoundedAlternativeSeminorms}
        The locally convex topology generated by the system
        \begin{equation}
            \label{eq:BoundedAlternativeSeminorms}
            \bigl\{
            \seminorm{r}_{B, \seminorm{q}, v_0}
            \colon
            \seminorm{q} \in \cs(W), \,
            v_0 \in U, \,
            B \in \mathfrak{B}(U,v_0)
            \bigr\}
        \end{equation}
        is the $\Holomorphic_\beta(U,W)$ topology. More precisely, we
        have the estimates
        \begin{equation}
            \seminorm{r}_{B, \seminorm{q}, v_0}
            \le
            \frac{\seminorm{p}_{v_0 + rB, \seminorm{q}}}{1 - r}
            \qquad \textrm{and} \qquad
            \seminorm{p}_{B, \seminorm{q}}
            \le
            \seminorm{r}_{B, \seminorm{q}}
        \end{equation}
        for any $B \in \Bdd(U,v_0)$ and $r > 0$ with $v_0 + rB
        \subseteq U$.

        \item \label{item:BoundedAlternativeEntireSeries}
        The seminorms corresponding to $\Bdd(V,v_0)$ in
        \eqref{eq:BoundedAlternativeSeminorms} generate the topology of
        $\Holomorphic_\beta(V,W)$ for every $v_0 \in V$.

        \item \label{item:BoundedAlternativeEntire}
        The seminorms corresponding to $\Bdd(V,v_0)$ in
        \eqref{eq:SeminormBounded} generate the topology of
        $\Holomorphic_\beta(V,W)$
        for every $v_0 \in V$.
    \end{propositionlist}
\end{proposition}
\begin{proof}
    Assume first that $f$ is bounded and let $v_0 \in U$, $B \in
    \Bdd(U,v_0)$ and $\seminorm{q} \in \cs(W)$. Applying the locally
    convex Cauchy estimates \eqref{eq:CauchyEstimatesLocallyConvex}, we
    get
    \begin{equation*}
        \seminorm{p}_{B, \seminorm{q}}
        (P_{n,v_0})
        \le
        \frac{\seminorm{p}_{v_0 + rB, \seminorm{q}}(f)}{r^n}
    \end{equation*}
    for any $0 < r < 1$ with $v_0 + rB \subseteq U$ and $n \in \N_0$.
    This yields the estimate
    \begin{equation*}
        \seminorm{r}_{B, \seminorm{q}, v_0}(f)
        =
        \sum_{n=0}^{\infty}
        \seminorm{p}_{B, \seminorm{q}}
        (P_{n,v_0})
        \le
        \frac{\seminorm{p}_{v_0 + rB, \seminorm{q}}(f)}{1 - r}.
    \end{equation*}
    Consequently, $\seminorm{r}_{B, \seminorm{q}, v_0}(f) < \infty$, as
    with $B$ also $v_0 + rB$ is bounded. Conversely, we know that the
    series \eqref{eq:TaylorSeries} converges on all of $B$ by our
    definition of $\Bdd(U,v_0)$ and thus we get
    \begin{equation*}
        \seminorm{p}_{B,\seminorm{q}}(f)
        \le
        \sum_{n=0}^{\infty}
        \seminorm{p}_{B,\seminorm{q}}
        (P_{n,v_0})
        =
        \seminorm{r}_{B, \seminorm{q}, v_0}(f).
    \end{equation*}
    This completes the proof of \ref{item:AlternativeBounded} and
    \ref{item:BoundedAlternativeSeminorms}. Moreover,
    \ref{item:AlternativeBoundedAreSeminorms} is now clear, as
    pointwisely convergent
    series over seminorms are seminorms. Finally, to see
    \ref{item:BoundedAlternativeEntireSeries} and
    \ref{item:BoundedAlternativeEntire},
    note that
    \begin{equation*}
        \Bdd(V,v_0)
        =
        \bigl\{
        B \subseteq V
        \colon
        B \;
        \textrm{is absolutely convex and bounded,} \;
        v_0 + B
        \subseteq
        V
        \bigr\}
    \end{equation*}
    is the collection of all bounded absolutely convex subsets
    of $V$, as the condition $v_0 + B \subseteq V$ is vacuous. Hence,
    the claims follow from the fact that absolutely convex hulls of
    bounded sets are bounded.
\end{proof}

%% file: TeX/TensorProducts.tex

For us, only projective and injective tensor products will play a role.
To fix our notation, and for the convenience of the reader, we briefly
recall their definitions and crucial properties. Staying in the category
of locally convex spaces, they are most easily defined by means of
explicit systems of continuous seminorms built from the continuous
seminorms of the tensor factors. Anticipating continuity of the linear
algebraic associativity identifications, it suffices to consider tensor
products of two seminorms. Indeed, if $V$ and $W$ are locally convex and
$\seminorm{q} \in \cs(V)$ and $\seminorm{p} \in \cs(W)$, then their
projective tensor product is defined as
\begin{equation}
    \bigl(
        \seminorm{q} \tensor_\pi \seminorm{p}
    \bigr)(x)
    \coloneqq
    \inf
    \biggl\{
        \sum_k
        \seminorm{q}
        \bigl(v_k\bigr)
        \cdot
        \seminorm{p}
        \bigl(w_k\bigr)
        \colon
        x
        =
        \sum_k
        v_k
        \tensor
        w_k
    \biggr\},
\end{equation}
where $x \in V \tensor W$. The infimum is thus taken over all
decompositions of $x$ into factorizing tensors. The most important
property of the projective tensor product is that it fulfils the
universal property familiar from finite-dimensional vector spaces, where
now all maps are continuous. This is known as the infimum argument, and
can e.g. be found in~\cite[Prop.~43.4]{treves:2006a}. From a practical
point of view, this means that it suffices to estimate multilinear
mappings acting on factorizing tensors, which turns out to be a drastic
technical simplification due to the upcoming
\eqref{eq:TensorOnFactorising}.

Keeping the notation from before, the injective tensor product of
$\seminorm{q}$ and $\seminorm{p}$ is defined by
\begin{equation}
    \label{eq:TensorInjective}
    \bigl(
        \seminorm{q} \tensor_\epsilon \seminorm{p}
    \bigr)(x)
    \coloneqq
    \sup_{\abs{v'} \le \seminorm{q}}
    \sup_{\abs{w'} \le \seminorm{p}}
    \abs[\Big]
    {
        \bigl(
            v' \tensor w'
        \bigr)(x)
    },
\end{equation}
where $v' \in V'$ and $w' \in W'$ are continuous linear
functionals and the ordering is defined pointwisely. As a first
property, we show that the suprema are taken over the polars of the unit
cylinders
\begin{equation}
    \Ball_{\seminorm{q},1}(0)
    \coloneqq
    \bigl\{
        v \in V
        \colon
        \seminorm{q}(v)
        <
        1
    \bigr\}
\end{equation}
and $\Ball_{\seminorm{p},1}(0) \subseteq W$. Recall that for a subset $A
\subseteq V$, its polar is
\begin{equation}
    \label{eq:Polar}
    A^\polar
    \coloneqq
    \bigl\{
        v' \in V'
        \colon
        \abs[\big]{v'(v)}
        \le
        1
        \textrm{ for all }
        v \in A
    \bigr\}.
\end{equation}
\begin{lemma}
    \label{lem:PolarsOfBallsBounded}
    Let $V$ be a locally convex space and $\seminorm{q} \in \cs(V)$.
    Then the polar
    \begin{equation}
        \label{eq:PolarOfCylinder}
        \Ball_{\seminorm{q},1}(0)^\polar
        =
        \bigl\{
            v' \in V'
            \colon
            \abs{v'}
            \le
            \seminorm{q}
        \bigr\}
        \subseteq
        V_\beta'
    \end{equation}
    is bounded with respect to the topology of uniform convergence on
    bounded sets.
\end{lemma}
\begin{proof}
    Let $v' \in \Ball_{\seminorm{q},1}(0)^\polar$. Unwrapping the
    definition, this means $\abs{v'(v)} \le 1$ for all $v \in V$ with
    $\seminorm{q}(v) < 1$. If now $v \in V$ fulfils $\seminorm{q}(v) >
    0$, this implies
    \begin{equation*}
        \abs[\big]
        {v'(v)}
        =
        r
        \cdot
        \seminorm{q}(v)
        \cdot
        \abs[\bigg]
        {
            v'
            \Bigl(
                \frac{v}{r \cdot \seminorm{q}(v)}
            \Bigr)
        }
        \le
        r
        \cdot
        \seminorm{q}(v)
    \end{equation*}
    for any $0 < r < 1$, i.e. $\abs{v'(v)} \le 1$. Moreover, if
    $\seminorm{q}(v) = 0$, then
    \begin{equation*}
        r
        \cdot
        \seminorm{q}(v)
        =
        \seminorm{q}(rv)
        =
        0
    \end{equation*}
    implies
    \begin{equation*}
        r
        \cdot
        \abs[\big]{v'(v)}
        =
        \abs[\big]{v'(r \cdot v)}
        \le
        1
    \end{equation*}
    for all $r > 0$, i.e. $v'(v) = 0 = \seminorm{q}(v)$. In both cases,
    $\abs{v'} \le \seminorm{q}$ follows. Conversely, if $v' \in V'$
    fulfils $\abs{v'} \le \seminorm{q}$ and $v \in
    \Ball_{\seminorm{q},1}(0)$, then
    \begin{equation*}
        \abs[\big]{v'(v)}
        \le
        \seminorm{q}(v)
        <
        1,
    \end{equation*}
    proving the other inclusion within \eqref{eq:PolarOfCylinder}. Let
    now $B \subseteq V$ be bounded and $\seminorm{q} \in \cs(V)$. By
    boundedness of the former, there exists some radius $r > 0$ with $B
    \subseteq \Ball_{\seminorm{q},r}(0)$. Thus,
    by~\eqref{eq:PolarOfCylinder}, we get
    \begin{equation*}
        \sup_{v' \in \Ball_{\seminorm{q},1}(0)^\polar}
        \seminorm{p}_B(v')
        =
        \sup_{v' \in \Ball_{\seminorm{q},1}(0)^\polar}
        \sup_{v \in B}
        \abs[\big]
        {v'(v)}
        \le
        \sup_{v \in B}
        \seminorm{q}(v)
        \le
        r
        <
        \infty,
    \end{equation*}
    showing the boundedness of $\Ball_{\seminorm{q},1}(0)^\polar$.
\end{proof}

In passing, we note that the unit cylinders themselves are typically
unbounded, as the kernel of the seminorm is a subspace contained in all
cylinders. Polars enjoy numerous pleasant properties and are central to
the locally convex incarnation of the Banach-Steinhaus Theorem, see e.g.
\cite[§20.8]{koethe:1969a} for a comprehensive treatment.

While the double supremum in \eqref{eq:TensorInjective} reflects the
symmetric role of both tensor factors nicely, it is often useful to
break this symmetry. Using the partial natural pairings
\begin{align*}
    &\iota_V
    \colon
    V \tensor W
    \longrightarrow
    L(V',W), \quad
    \iota_V(v \tensor w)v'
    \coloneqq
    v'(v) \cdot w, \\
    &\iota_W
    \colon
    V \tensor W
    \longrightarrow
    L(W',V), \quad
    \iota_W(v \tensor w)w'
    \coloneqq
    w'(w) \cdot v,
\end{align*}
we may write
\begin{equation}
    \label{eq:InjectiveOneSupremum}
    \bigl(
    \seminorm{q} \tensor_\epsilon \seminorm{p}
    \bigr)(x)
    =
    \sup_{\abs{v'} \le \seminorm{q}}
    \seminorm{p}
    \bigl(
    \iota_V(x)v'
    \bigr)
    =
    \sup_{\abs{w'} \le \seminorm{p}}
    \seminorm{q}
    \bigl(
    \iota_W(x)w'
    \bigr)
\end{equation}
for $x \in V \tensor W$, $\seminorm{q} \in \cs(V)$ and $\seminorm{p} \in
\cs(W)$. To see this, recall
\begin{equation*}
    \seminorm{q}(v)
    =
    \sup_{\abs{v'} \le \seminorm{q}}
    \abs[\big]{v'(v)}
\end{equation*}
by the Hahn-Banach Theorem. Decomposing $x = \sum_k v_k \tensor w_k$ thus
yields
\begin{align*}
    \sup_{\abs{v'} \le \seminorm{q}}
    \seminorm{p}
    \bigl(
        \iota_V(x)v'
    \bigr)
    &=
    \sup_{\abs{v'} \le \seminorm{q}}
    \sup_{\abs{w'} \le \seminorm{p}}
    \abs[\bigg]
    {
        w'
        \biggl(
        \sum_k
        v'(v_k)
        \cdot
        w_k
        \biggr)
    } \\
    &=
    \sup_{\abs{v'} \le \seminorm{q}}
    \sup_{\abs{w'} \le \seminorm{p}}
    \abs[\bigg]
    {
        \sum_k
        v'(v_k)
        \cdot
        w'(w_k)
    } \\
    &=
    \bigl(
    \seminorm{q}
    \tensor_\epsilon
    \seminorm{p}
    \bigr)(x)
\end{align*}
as desired. The other equality in \eqref{eq:InjectiveOneSupremum}
follows analogously.

Another seemingly trivial yet important fact is the following, which
makes it convenient to study both injective and projective tensor
products in tandem.
\begin{lemma}[Injective vs. projective $\tensor$, {\cite[Cor. of
        Prop.~43.4]{treves:2006a}}]
    \label{lem:ProjectiveVsInjective}
    Let $V$ and~$W$ be locally convex spaces. Then the identity mapping
    \begin{equation}
        \label{eq:ProjectiveVsInjectiveIdentity}
        V \tensor_\pi W
        \longrightarrow
        V \tensor_\epsilon W
    \end{equation}
    is continuous. More precisely, if $\seminorm{q} \in \cs(V)$ and
    $\seminorm{p} \in \cs(W)$, then
    \begin{equation}
        \label{eq:ProjectiveVsInjective}
        \seminorm{q} \tensor_\epsilon \seminorm{p}
        \le
        \seminorm{q} \tensor_\pi \seminorm{p}.
    \end{equation}
\end{lemma}
\begin{proof}
   If $\abs{v'} \le \seminorm{q}$, $\abs{w'} \le \seminorm{p}$ and $x =
   \sum_k v_k \tensor w_k$, then
    \begin{equation*}
        \abs{(v' \tensor w')x}
        =
        \abs[\bigg]
        {
            \sum_k
            v'(v_k)
            \cdot
            w'(w_k)
        }
        \le
        \sum_k
        \abs[\big]
        {v'(v_k)}
        \cdot
        \abs[\big]
        {w'(w_k)}
        \le
        \sum_k
        \seminorm{q}(v_k)
        \cdot
        \seminorm{p}(w_k).
    \end{equation*}
    Taking the infimum over all decompositions of $x$ and the suprema
    over the polars yields~\eqref{eq:ProjectiveVsInjective}.
\end{proof}

Combining the Hahn-Banach Theorem with \eqref{eq:ProjectiveVsInjective},
now readily implies the following well-known pleasant property of both
tensor products on factorizing tensors.
\begin{corollary}
    Let $V$ and $W$ be locally convex spaces and $\seminorm{q} \in
    \cs(V)$ as well as $\seminorm{p} \in \cs(W)$. Then
    \begin{equation}
        \label{eq:TensorOnFactorising}
        \bigl(
            \seminorm{q} \tensor_\pi \seminorm{p}
        \bigr)(v \tensor w)
        =
        \seminorm{q}(v)
        \cdot
        \seminorm{p}(w)
        =
        \bigl(
            \seminorm{q} \tensor_\epsilon \seminorm{p}
        \bigr)(v \tensor w)
    \end{equation}
    for $v \in V$ and $w \in W$.
\end{corollary}

Following the same line of reasoning, one moreover establishes that both
flavours of tensor products inherit the Hausdorff property from their
factors. Moreover, tensor products of norms constitute norms.

Finally, we note that taking successive projective or injective tensor
products leads to continuous mappings encoding the associativity
of the algebraic tensor product. Indeed, taking successive tensor
products of seminorms simply leads to the same seminorms on the triple
tensor product. Hence, it makes sense to speak of projective and
injective tensor powers $\Tensor^n_\pi(V)$ and $\Tensor_\epsilon^n(V)$,
respectively. We take a moment to make this discussion precise for the
injective tensor product, as this leads to a useful explicit formula.
\begin{corollary}
    \label{lem:InjectiveTripleProduct}%
    Let $V,W$ and $X$ be locally convex spaces as well as $\seminorm{q}
    \in \cs(V)$, $\seminorm{p} \in \cs(W)$ and $r \in \cs(X)$. Then
    \begin{equation}
        \label{eq:InjectiveTripleProduct}
        \bigl(
        \seminorm{q}
        \tensor_\epsilon
        \seminorm{p}
        \bigr)
        \tensor_\epsilon
        r
        \at[\Big]{z}
        =
        \sup_{\abs{v'} \le \seminorm{q}}
        \sup_{\abs{w'} \le \seminorm{p}}
        \sup_{\abs{x'} \le \seminorm{r}}
        \abs[\Big]
        {
            \bigl(
            v' \tensor w' \tensor x'
            \bigr)(z)
        }
        =
        \seminorm{q}
        \tensor_\epsilon
        \bigl(
        \seminorm{p}
        \tensor_\epsilon
        r
        \bigr)
        \at[\Big]{z}
    \end{equation}
    for $z \in V \tensor W \tensor X$. In particular,
    \begin{equation}
        \bigl(
        V \tensor_\epsilon W
        \bigr)
        \tensor_\epsilon
        X
        \cong
        V
        \tensor_\epsilon
        \bigl(
        W \tensor_\epsilon X
        \bigr)
    \end{equation}
    as locally convex spaces.
\end{corollary}
\begin{proof}
    Use \eqref{eq:InjectiveOneSupremum} to expand the left-hand side
    while keeping $\seminorm{q} \tensor_\epsilon \seminorm{p}$. Plugging
    in \eqref{eq:TensorInjective} then leads to the middle expression in
    \eqref{eq:InjectiveTripleProduct}.
\end{proof}

This concludes our brief and incomplete review. Systematic discussions of
projective tensor products can be found in the textbooks
\cite[§41]{koethe:1979a} and \cite[Ch.~15]{jarchow:1981a}, and
\cite[Sec.~6.3]{vogt:2000a} contains a modern locally convex exposition
of injective ones. In particular, the crucial simplification
\eqref{eq:InjectiveOneSupremum} is \cite[6.28~Lemma]{vogt:2000a}.

%% file: TeX/RTopologies.tex

Let $V$ be a locally convex space. Fixing a parameter $R \ge 0$ and
$x \in \{\pi,\epsilon\}$, we define seminorms
\begin{equation}
    \label{eq:SeminormRTopology}
    \seminorm{q}_{r}^{(R,x)}(v)
    \coloneqq
    \sum_{n=0}^{\infty}
    n!^R
    \cdot
    r^n
    \cdot
    \seminorm{q}^{\tensor_{x} n}(v_n),
\end{equation}
where $v = \sum_{n=0}^{\infty} v_n \in \Tensor^\bullet(V)$ with
homogeneous components $v_n \in \Tensor^n(V)$ for $n \in \N_0$,
and we vary the continuous seminorm $\seminorm{q} \in \cs(V)$ and $r \ge
0$. Notably, the \eqref{eq:SeminormRTopology} is only formally infinite,
as every tensor is a finite linear combination of homogeneous ones. The
injective, respectively projective, $R$-topology is
then defined as the locally convex topology induced by the seminorms
\eqref{eq:SeminormRTopology} for $\seminorm{q} \in \cs(V)$ and $r=1$.
Note that the other seminorms $\seminorm{q}_r^{(R,x)}$ are then
automatically continuous, as with $\seminorm{q}$ also its multiples
constitute continuous seminorms. The additional index is however useful
for computations. We denote the resulting locally convex spaces by
$\Tensor_{R,x}^\bullet(V)$. In Section~\ref{sec:TensorsAsPolynomials}, we
will work with the symmetric tensor algebra, endowed with the subspace
topologies, for which we then write $\Sym_{R,x}^\bullet(V)$.
In \cite[Sec.~3.1 \& 4]{waldmann:2014a}, the
algebra~$\Tensor_{R,\pi}^\bullet(V)$ was introduced and studied in
detail. We
provide the analogous discussion for $\Tensor_{R,\epsilon}^\bullet(V)$.
\begin{proposition}
    Let $V$ be a locally convex space and $R,r \ge 0$.
    \begin{propositionlist}
        \item For every $n \in \N_0$, the projection and inclusion
        \begin{equation}
            \Tensor_{R,\epsilon}^\bullet(V)
            \twoheadrightarrow
            \Tensor^n_{\epsilon}(V)
            \hookrightarrow
            \Tensor_{R,\epsilon}^\bullet(V)
        \end{equation}
        are continuous linear mappings.

        \item If $V$ is Hausdorff, then so is
        $\Tensor_{R,\epsilon}^\bullet(V)$.

        \item The tensor product is continuous on
        $\Tensor_{R,\epsilon}^\bullet(V)$. More precisely,
        \begin{equation}
            \seminorm{q}_{r}^{(R,\epsilon)}
            (v \tensor w)
            \le
            \seminorm{q}_{2^R r}^{(R,\epsilon)}(v)
            \cdot
            \seminorm{q}_{2^R r}^{(R,\epsilon)}(w)
            \qquad
            \textrm{for }
            v,w \in \Tensor_{R,\epsilon}^\bullet(V).
        \end{equation}

        \item The symmetric tensor product is continuous on
        $\Sym_{R,\epsilon}^\bullet(V)$. More precisely,
        \begin{equation}
            \seminorm{q}_{r}^{(R,\epsilon)}
            (v \vee w)
            \le
            \seminorm{q}_{2^R r}^{(R,\epsilon)}(v)
            \cdot
            \seminorm{q}_{2^R r}^{(R,\epsilon)}(w)
            \qquad
            \textrm{for }
            v,w \in \Sym_{R,\epsilon}^\bullet(V).
        \end{equation}

        \item \label{item:InjectiveCompletion}%
        Let $V$ be Hausdorff. The completion
        $\widehat{\Tensor}_{R,\epsilon}^\bullet(V)$ of
        $\Tensor_{R,\epsilon}^\bullet(V)$ is given by
        \begin{equation}
            \label{eq:InjectiveCompletion}
            \widehat{\Tensor}_{R,\epsilon}^\bullet(V)
            \cong
            \biggl\{
                v
                =
                \sum_{n=0}^\infty
                v_n
                \colon
                \seminorm{q}_{r}^{(R,\epsilon)}(v)
                <
                \infty
                \textrm{ for }
                \seminorm{q} \in \cs(V),
                r \ge 0
            \biggr\}
            \subseteq
            \prod_{n=0}^\infty
            \widehat{\Tensor}_{\epsilon}^n(V),
        \end{equation}
        where we extend $\seminorm{q}_{r}^{(R,\epsilon)}$ to the
        Cartesian product by allowing the value $\infty$.

        \item The identity map \eqref{eq:ProjectiveVsInjectiveIdentity}
        extends to a continuous linear injection
        \begin{equation}
            \widehat{\Tensor}_{R,\pi}^\bullet(V)
            \hookrightarrow
            \widehat{\Tensor}_{R,\epsilon}^\bullet(V).
        \end{equation}
    \end{propositionlist}
\end{proposition}
\begin{proof}
    All claims are routine verifications along the same
    lines as \cite[Sec.~3.1]{waldmann:2014a}. Note that the universal
    property of the projective tensor product was not important for any
    of the arguments. Instead, the crucial ingredients for the
    continuity estimates are~\eqref{eq:TensorOnFactorising} and
    \eqref{eq:InjectiveTripleProduct}, and the fact that one may reorder
    series of non-negative real numbers.
\end{proof}

An important consequence of \ref{item:InjectiveCompletion} is that if
$v_n \in \Tensor_{\epsilon}^n(V)$ are such that
\begin{equation}
    v
    =
    \sum_{n=0}^{\infty}
    v_n
    \in
    \widehat{\Tensor}_{R,\epsilon}^\bullet(V),
\end{equation}
then the series actually converges within the
$\Tensor_{R,\epsilon}$-topology. This justifies the abuse of notation in
\eqref{eq:InjectiveCompletion}. In Corollary~\ref{cor:TensorAsFrechet},
we will interpret this series as the Taylor expansion of the holomorphic
mapping $\widehat{\iota}(v)$ from the introduction.

%% file: TeX/TensorsAsPolynomials.tex

Having both polynomials and the symmetric algebra in mind, we may
connect both. Recall that the symmetric tensor $v_1 \vee \cdots \vee v_n
\in \Sym^n(V)$ with $v_1,\ldots,v_n \in V$ induces an $n$-homogeneous
polynomial
\begin{equation}
    V'
    \ni
    v'
    \quad \mapsto \quad
    v'(v_1) \cdots v'(v_n)
    \coloneqq
    \iota(v_1 \vee \cdots \vee v_n)v'
    \in
    \C.
\end{equation}
In analogy to indicating the topology of uniform
convergence on bounded sets with a subscript $\beta$, we shall in the
sequel use the subscript $\sigma$ to indicate the topology of
pointwise convergence. We also refer to the
$\beta$-topology as strong and to the~$\sigma$-topology as weak.
\begin{lemma}
    \label{lem:ImageOfIota}%
    Let $V$ be a Hausdorff locally convex space, $n \in \field{N}$ and
    $v_1,\ldots,v_n \in V$. Then the polynomial
    \begin{equation}
        \iota(v_1 \vee \cdots \vee v_n)
        \colon
        V'_\sigma
        \longrightarrow
        \field{C}
    \end{equation}
    is a homogeneous polynomial of degree $n$ of finite
    type and continuous. In particular, it is continuous as a polynomial
    on $V'_\beta$.
\end{lemma}

Recall that a polynomial $P \in \Pol^n(V'_\beta)$ is of \emph{finite
    type} if there are $a_1, \ldots, a_N \in (V'_\beta)'$ and
    corresponding coefficients $\lambda_1, \ldots, \lambda_N \in \C$
    such that
\begin{equation}
    P(v')
    =
    \sum_{k=0}^{N}
    \lambda_k
    \cdot
    a_k(v')^n
    \qquad
    \textrm{for }
    v' \in V'.
\end{equation}
Note that this linear combination takes place within fixed homogeneity.
\begin{proof}[Of Lemma~\ref{lem:ImageOfIota}]
    Let $n \in \field{N}$ and $v_1, \ldots, v_n \in V$. Note that the
    seminorm
    \begin{equation*}
        \seminorm{q}(v')
        \coloneqq
        \max_{j=1, \ldots, n}
        \abs[\big]{v'(v_j)}
    \end{equation*}
    is continuous on $V'_\sigma$. Thus the estimate
    \begin{equation*}
        \abs[\big]
        {\iota(v_1 \vee \cdots \vee v_n)v'}
        =
        \prod_{j=1}^n
        \abs[\big]
        {v'(v_j)}
        \le
        \prod_{j=1}^n
        \seminorm{q}(v')
        =
        \seminorm{q}(v')^n
        \qquad
        \textrm{for }
        v' \in V'
    \end{equation*}
    proves the $\sigma$-continuity of $\iota(v_1 \vee \cdots \vee v_n)$.
    As the strong topology is finer than the weak topology, the
    polynomial $\iota(v_1 \vee \cdots \vee v_n)$ is thus
    also continuous on $V_\beta'$. By the polarization identity
    \eqref{eq:Polarization},
    we may finally write
    \begin{equation*}
        v'(v_1) \cdots v'(v_n)
        =
        \frac{1}{2^n \cdot n!}
        \sum_{\epsilon_j = \pm 1}
        \epsilon_1 \cdots \epsilon_n
        \cdot
        v'
        \biggl(
        \sum_{k=1}^{n}
        \epsilon_k v_k
        \biggr)^n
    \end{equation*}
    for $v' \in V'$, as the left-hand side is linear and symmetric
    with respect to $v_1 \tensor \cdots \tensor v_n$. Consequently,
    the index set $J \coloneqq \{-1,1\}^n$, viewed as maps
    $\{1,\ldots,n\} \longrightarrow \{-1,1\}$,
    \begin{equation}
        a_j
        \coloneqq
        \sum_{k=1}^{n}
        j(k)
        v_k
        \qquad \textrm{and} \qquad
        \lambda_j
        \coloneqq
        \frac{j(1) \cdots j(n)}{2^n \cdot n!}
    \end{equation}
    for all $j \in J$ constitutes the desired finite-type decomposition.
\end{proof}

Hence, our observation assembles to a linear mapping
\begin{equation}
    \label{eq:Iota}
    \iota
    \colon
    \Sym^\bullet(V)
    \longrightarrow
    \Pol(V'_\beta), \quad
    \iota(v_1 \vee \cdots \vee v_n)
    \at[\Big]{v'}
    \coloneqq
    \prod_{j=1}^n
    v'(v_j)
\end{equation}
with components
\begin{equation}
    \iota_n
    \coloneqq
    \iota
    \at[\Big]{\Sym^n(V)}
    \colon
    \Sym^n(V)
    \longrightarrow
    \Pol^n(V')
    \qquad \textrm{for }
    n \in \N.
\end{equation}
In zeroth degree, we set
\begin{equation}
    \iota_0
    \colon
    \Sym^0(V)
    =
    \C
    \longrightarrow
    \Pol^{0}(V'),
    \quad
    \iota_0(z)v'
    \coloneqq
    z.
\end{equation}
In passing, we note that with these conventions, $\iota$ is actually a
morphisms of algebras, translating the symmetric tensor product into the
pointwise product of mappings. While this point of view is sometimes
useful, we will use it sparingly in the sequel. Instead, we proceed with
an algebraic preliminary consideration and prove the injectivity of
$\iota$.
\begin{lemma}
    \label{lem:IotaInjectivity}%
    Let $V$ be a Hausdorff locally convex space. Then the
    mapping~$\iota$ from \eqref{eq:Iota} and its components $\iota_n$
    are injective for $n \in \N_0$.
\end{lemma}
\begin{proof}
    As polynomials of different homogeneous degrees are linearly
    independent, it suffices to prove the injectivity of the component
    maps $\iota_n$. For $n = 0$ there is nothing to be shown. Thus
    assume $n \in \N$ and let
    \begin{equation*}
        0
        \neq
        v \in \Sym^n(V).
    \end{equation*}
    We choose a linear algebraic basis $\{\basis{e}_\alpha \colon \alpha
    \in J \} \subseteq V$ of $V$ to expand
    \begin{equation*}
        v
        =
        \sum_{\alpha_1,\ldots,\alpha_n \in J}
        v^{\alpha_1 \cdots \alpha_n}
        \cdot
        \basis{e}_{\alpha_1} \vee \cdots \vee \basis{e}_{\alpha_n}.
    \end{equation*}
    As~$v$ is nonzero, there are indices $\alpha_1, \ldots, \alpha_n
    \in J$ such that the corresponding coefficient~$v^{\alpha_1 \cdots
    \alpha_n}$ is nonzero, as well. Moreover, if $\beta_1, \ldots,
    \beta_n \in J$ is a different choice of indices, in the sense that
    $(\beta_1,\ldots, \beta_n)$ is not a permutation of $(\alpha_1,
    \ldots, \alpha_n)$, such that the corresponding
    coefficient $v^{\beta_1 \cdots \beta_n} \neq 0$, then
    \begin{equation*}
        \beta_j \notin \{\alpha_1, \ldots,
        \alpha_n\}
        \qquad
        \textrm{for some}
        \quad
        j \in \{1, \ldots, n\}.
    \end{equation*}
    Choosing one such index for every other non-vanishing coefficient
    defines a finite index set $J_0 \subseteq J$, as we are working with
    a linear algebraic basis. Consider the linear functional
    \begin{align*}
        &v'
        \colon
        \Span
        \bigl(
        \bigl\{
            \basis{e}_{\alpha_1}, \ldots, \basis{e}_{\alpha_n}
        \bigr\}
        \cup
        \{\basis{e}_\beta \colon \beta \in J_0\}
        \bigr)
        \longrightarrow
        \field{C}, \\
        &v'
        \biggl(
        \sum_{k=1}^n
        \eta^k
        \basis{e}_{\alpha_k}
        +
        \sum_{\beta \in J_0}
        \lambda^\beta
        \basis{e}_{\beta}
        \biggr)
        \coloneqq
        \sum_{k=1}^{n}
        \eta^k.
    \end{align*}
    As the various $\basis{e}_\beta$ are linearly independent, our
    functional $v'$ is well defined. Moreover, as its domain is finite
    dimensional and Hausdorff, $v'$ is continuous. By the Hahn-Banach
    Theorem, we find a continuous extension $w' \in V'$ of $v'$.
    By construction,
    \begin{equation*}
        \iota_n(v)
        \at[\Big]{w'}
        =
        \sum_{\gamma_1,\ldots,\gamma_n \in J}
        v^{\gamma_1 \cdots \gamma_n}
        \cdot
        w'(\basis{e}_{\gamma_1})
        \cdots
        w'(\basis{e}_{\gamma_n})
        =
        n!
        \cdot
        v^{\alpha_1 \cdots \alpha_n}
        \neq
        0,
    \end{equation*}
    as every other term contains a factor $0$. This proves that the
    linear mapping $\iota_n$ is injective and thus the same is true for
    $\iota$ itself, as we are dealing with direct sums throughout.
\end{proof}

Next, we study the continuity of $\iota$ and its components, first for
injective tensor products.
\begin{theorem}[Tensors as Polynomials I]
    \label{thm:TensorsAsPolynomials}%
    Let $V$ be a Hausdorff locally convex space.
    \begin{theoremlist}
        \item \label{item:IotaNSigmaContinuous}%
        The restrictions
        \begin{equation}
            \label{eq:IotaNSigmaContinuous}
            \iota_n
            \colon
            \Sym^n_\epsilon(V)
            \longrightarrow
            \Pol_\sigma^n
            (V'_\beta)
        \end{equation}
        are continuous for $n \in \N_0$.

        \item \label{item:IotaNBetaContinuous}%
        If $V$ is barrelled, then the restrictions
        \begin{equation}
            \iota_n
            \colon
            \Sym^n_\epsilon(V)
            \longrightarrow
            \Pol_\beta^n(V'_\beta)
        \end{equation}
        are continuous for $n \in \N_0$.

        \item \label{item:IotaContinuous}%
        If $V$ is barrelled and $R \ge 0$, then the mapping
        \begin{equation}
            \label{eq:IotaContinuous}
            \iota
            \colon
            \Sym_{R,\epsilon}^\bullet(V)
            \longrightarrow
            \Pol^\bullet_\beta
            (V'_\beta)
        \end{equation}
        is continuous.
    \end{theoremlist}
\end{theorem}
\begin{proof}
    The case $n=0$ is clear. Let $n \in \N$, $v \in \Sym^n(V)$ and
    $\varphi \in V'$ with corresponding seminorm
    $\seminorm{p}_\varphi(Q) \coloneqq \abs{Q(\varphi)}$ for $Q \in
    \Pol^\bullet_\sigma(V'_\beta)$. By continuity of $\varphi$, we
    moreover have
    \begin{equation}
        \seminorm{q}_\varphi
        \coloneqq
        \abs{\varphi} \in
        \cs(V).
    \end{equation}
    Using the higher order generalization of
    \eqref{eq:InjectiveTripleProduct} yields the estimate
    \begin{equation*}
        \label{eq:IotaProof}
        \seminorm{p}_\varphi
        \bigl(
        \iota_n(v)
        \bigr)
        =
        \abs[\big]
        {
            \varphi^{\tensor n}(v)
        }
        \le
        \sup_{\abs{v'_1}, \ldots, \abs{v_n'} \le \abs{\varphi}}
        \abs[\Big]
        {
            \bigl(
                v_1' \tensor \cdots \tensor v_n'
            \bigr)(v)
        }
        =
        \seminorm{q}_\varphi^{\tensor_\epsilon n}(v),
        \tag{$\ast$}
    \end{equation*}
    proving the continuity of \eqref{eq:IotaNSigmaContinuous}. Assume
    now that $V$ is barrelled and let $B' \subseteq V'_\beta$ be
    a bounded subset. Consider
    \begin{equation*}
        \seminorm{q}
        \coloneqq
        \sup_{\varphi \in B'}
        \seminorm{q}_\varphi,
    \end{equation*}
    which is a well defined, continuous seminorm on $V$ by the
    Banach-Steinhaus Theorem: as a pointwise supremum of continuous
    mappings it is lower semicontinuous and the Banach-Steinhaus Theorem
    yields its continuity. By what we have already argued, we
    get~\eqref{eq:IotaProof} for every $\varphi \in B'$, and thus arrive
    at the continuity estimate
    \begin{equation*}
        \seminorm{p}_{B'}
        \bigl(
            \iota_n(v)
        \bigr)
        =
        \sup_{\varphi \in B'}
        \seminorm{p}_\varphi
        \bigl(
            \iota_n(v)
        \bigr)
        \le
        \sup_{\varphi \in B'}
        \seminorm{q}_\varphi^{\tensor_\epsilon n}(v)
        \le
        \seminorm{q}^{\tensor_\epsilon n}(v),
    \end{equation*}
    as $\abs{v'} \le \abs{\varphi}$ for some $\varphi \in B'$ implies
    $\abs{v'} \le \seminorm{q}$ for $v' \in V'$. Finally turning to
    \eqref{eq:IotaContinuous}, we have
    \begin{equation*}
        \seminorm{p}_{B'}
        \bigl(
        \iota(v)
        \bigr)
        \le
        \sum_{n=0}^\infty
        \seminorm{p}_{B'}
        \bigl(
        \iota_n(v_n)
        \bigr)
        \le
        \sum_{n=0}^\infty
        \seminorm{q}^{\tensor_\epsilon n}
        \bigl(v_n)
        =
        \seminorm{q}_{1}^{(0,\epsilon)}
        \bigl(\iota(v)\bigr)
        \le
        \seminorm{q}_{1}^{(R,\epsilon)}
        \bigl(\iota(v)\bigr)
    \end{equation*}
    for $v = \sum_{n=0}^{\infty} v_n \in \Sym_R^\bullet(V)$ with
    homogeneous components $v_n \in \Sym^n(V)$ for $n \in \N_0$.
\end{proof}

By continuity of \eqref{eq:ProjectiveVsInjectiveIdentity}, our
considerations also prove a projective version of
Theorem~\ref{thm:TensorsAsPolynomials}.
\begin{corollary}[Tensors as Polynomials II]
    \label{cor:TensorsAsPolynomials}%
    Let $V$ be a Hausdorff locally convex space.
    \begin{theoremlist}
        \item \label{item:IotaNSigmaContinuousPi}%
        The restrictions
        \begin{equation}
            \label{eq:IotaNSigmaContinuousPi}
            \iota_n
            \colon
            \Sym^n_\pi(V)
            \longrightarrow
            \Pol_\sigma^n
            (V'_\beta)
        \end{equation}
        are continuous for $n \in \N_0$.

        \item \label{item:IotaNBetaContinuousPi}%
        If $V$ is barrelled, then the restrictions
        \begin{equation}
            \iota_n
            \colon
            \Sym^n_\pi(V)
            \longrightarrow
            \Pol_\beta^n(V'_\beta)
        \end{equation}
        are continuous for $n \in \N_0$.

        \item \label{item:IotaContinuousPi}%
        If $V$ is barrelled and $R \ge 0$, then the mapping
        \begin{equation}
            \iota
            \colon
            \Sym_{R,\pi}^\bullet(V)
            \longrightarrow
            \Pol^\bullet_\beta
            (V'_\beta)
        \end{equation}
        is continuous.
    \end{theoremlist}
\end{corollary}

Our next goal is to investigate the flavour of holomorphic functions the
unique continuous linear extensions
\begin{equation}
    \widehat{\iota}
    \colon
    \widehat{\Sym}_{0,\epsilon}(V)
    \longrightarrow
    \Bounded(V_\beta')
    \qquad \textrm{and} \qquad
    \widehat{\iota}
    \colon
    \widehat{\Sym}_{0,\pi}(V)
    \longrightarrow
    \Bounded(V_\beta')
\end{equation}
of \eqref{eq:Iota} produce. Here, we have used the completeness of
$\Bounded(V_\beta') \coloneqq \Bounded(V_\beta',\C)$ from
Proposition~\ref{prop:BoundedCompleteness}. By slight abuse of notation,
we denote both extensions by the same symbol. This should not cause too
much confusion. Observe now the following.
\begin{lemma}
    \label{lem:IotaExtension}%
    Let $V$ be a barrelled Hausdorff locally convex space and
    \begin{equation}
        v
        =
        \sum_{n=0}^\infty
        v_n
        \in
        \widehat{\Sym}_{0,\pi}^\bullet(V)
    \end{equation}
    with homogeneous components $v_n \in \widehat{\Sym}^n_\pi(V)$. Then
    \begin{equation}
        \label{eq:IotaExtension}
        \widehat{\iota}
        (v)
        =
        \sum_{n=0}^\infty
        \widehat\iota(v_n)
        =
        \sum_{n=0}^\infty
        \widehat\iota_n(v_n),
    \end{equation}
    where $\widehat\iota_n \colon \widehat{\Sym}^n_\pi(V)
    \longrightarrow \Bounded(V'_\beta)$ denotes the unique linear
    continuous extension of $\iota_n$.
\end{lemma}

As the $\epsilon$-version of the topology is coarser, the statement also
holds in this setting. The same can be said about the weighted
topologies for any $R \ge 0$.
\begin{proof}[Of Lemma~\ref{lem:IotaExtension}]
    The crucial point is that the partial sums of the series
    $\sum_{n=0}^\infty v_n$ actually converge to $v$ in
    the~$\Sym_{0,\pi}$-topology. This shows the first equality in
    \eqref{eq:IotaExtension} by linearity and continuity. The second
    follows from
    \begin{equation}
        \widehat\iota_n
        =
        \widehat{\iota}
        \at[\Big]
        {\Sym^n(V)},
    \end{equation}
    as both maps are continuous and agree on the dense subspace
    $\Sym^n(V)$.
\end{proof}

In passing, we observe that already the homogeneous components~$v_n$
might not be linear combinations of factorizing tensors any more. Thus
there are two limiting procedures to be understood: a \emph{horizontal}
completion within each homogeneous degree and a \emph{vertical}
completion of the symmetric algebra, where the summability condition
comes into play.

\begin{remark}[DF-spaces]
    \label{rem:DF}%
    Multilinearity is preserved under pointwise limits, and thus the
    horizontal completion in each homogeneous degree consists of not
    necessarily continuous polynomials. The situation simplifies if we
    can ensure the completeness of $\Pol_\beta(V_\beta')$, as in this
    case the extended components $\widehat{\iota}_n$ map again into
    $\Pol_\beta(V_\beta')$, i.e. we actually stay in the realm of
    continuous polynomials. Taking another look at
    Proposition~\ref{prop:CompletenessHomogeneousPolynomials}, we may
    ensure this by assuming first countability of $V'_\beta$, which is
    however a rather atypical behaviour for a strong dual space. Indeed,
    this means that we can find a fundamental sequence of bounded
    subsets of $V$. Together with our assumption of barelledness, this
    implies that $V$ is a \emph{DF-space}. Here, the abbreviation DF
    alludes to ``predual of Fréchet space'' and indeed, strong duals of
    DF-spaces are always Fréchet by \cite[§29.2~(1)]{koethe:1969a}.
    Thus, Hausdorff DF-spaces constitute the natural setting
    for well-behaved horizontal completions. A systematic treatment of
    DF-spaces can be found e.g. in the textbook
    \cite[§29.3]{koethe:1969a}.
\end{remark}

Reflexivity then gives rise to a remarkable class of examples.
\begin{example}[Reflexive nuclear Fréchet spaces]
    \label{ex:NuclearFrechet}%
    Let $F$ be a nuclear reflexive Fréchet space such as:
    \begin{examplelist}
        \item The algebra of smooth functions $\Cinfty(\Omega)$ defined
        on some open subset~$\emptyset \neq \Omega \subseteq \R^n$ with
        values in $\C$, see \cite[Example~28.9~(1)]{meise.vogt:1992a}
        and \cite[Prop~36.10]{treves:2006a}.

        \item The Schwartz space $\Schwartz(\R^n)$ of complex valued
        rapidly decreasing smooth functions on $\R^n$, see
        \cite[Prop~36.10 \& Thm.~51.5]{treves:2006a}.

        \item The algebra of holomorphic functions
        $\Holomorphic(\Omega)$ defined on
        some planar domain~$\Omega \subseteq \C^n$ with values
        in $\C$, see \cite[Example~28.9~(4)]{meise.vogt:1992a} and
        \cite[Corollary of Prop.~36.9 or Prop.~36.10]{treves:2006a},
        both of which are applicable by Montel's Theorem.
    \end{examplelist}

    Consider $V \coloneqq F'_\beta$. Then, by reflexivity, $V' =
    (F'_\beta)_\beta' \cong F$ is Baire as a Fréchet space. Moreover,
    $V$ is nuclear as the strong dual of a nuclear space by
    \cite[Prop.~50.6]{treves:2006a}. Finally,~$F$ is Montel as a nuclear
    Fréchet space by \cite[Prop.~50.2]{treves:2006a} and thus its strong
    dual $F'_\beta$ is also Montel by \cite[Prop.~36.10]{treves:2006a},
    which in particular implies that $F'_\beta$ is barrelled itself by
    \cite[Cor.~of Prop~36.9 \& Prop.~36.4]{treves:2006a}. Thus all
    assumptions of Theorem~\ref{thm:TensorsAsPolynomials} and the
    forthcoming Theorem~\ref{thm:IotaEmbedding1} and
    Theorem~\ref{thm:IotaEmbedding2} are fulfilled. Moreover, we may
    apply both Proposition~\ref{prop:CompletenessHomogeneousPolynomials}
    and Theorem~\ref{thm:CompletenessBoundedFrechet} to $V'_\beta = F$.
\end{example}

Returning to Lemma~\ref{lem:IotaExtension} in our new setting, our tools
readily prove Fréchet holomorphy.
\begin{corollary}
    \label{cor:TensorAsFrechet}%
    Let $V$ be a Hausdorff barrelled DF-space and $v \in
    \widehat{\Sym}_{0,\epsilon}(V)$. Then
    \begin{equation}
        \widehat\iota(v)
        \colon
        V_\beta'
        \longrightarrow
        \C
    \end{equation}
    is Fréchet holomorphic with Taylor polynomials at the origin given by
    $\widehat\iota_n(v_n)$ for $n \in \N_0$.
\end{corollary}
\begin{proof}
    By Remark~\ref{rem:DF}, the series \eqref{eq:IotaExtension}
    is a decomposition of $\iota(v)$ into homogeneous
    polynomials and thus constitutes its Taylor expansion by the
    uniqueness of the latter, see again \cite[Sec.~3.1]{dineen:1999a}.
    Hence, $\widehat{\iota}(v)$ is G\^{a}teaux holomorphic. Invoking
    Proposition~\ref{prop:FrechetBaire}, the continuity of its Taylor
    polynomials finally implies continuity of $\widehat{\iota}(v)$.
\end{proof}

Our main result is now that $\iota$ constitutes a topological embedding
when working with injective tensor products. That is to say, the
subspace topology induced by the inclusion
\begin{equation}
    \widehat{\iota}
    \bigl(
        \Sym_{0,\epsilon}(V)
    \bigr)
    \subseteq
    \Holomorphic_\beta(V'_\beta)
\end{equation}
matches with the original one. In view of
Lemma~\ref{lem:IotaInjectivity} this identifies the
$\Sym_{0,\epsilon}$-topology as nothing else than the topology of
uniform convergence on bounded sets. We have already taken care of the
brunt of the proof by establishing
Proposition~\ref{prop:BetaAlternativeSeminorms}.
\begin{theorem}[Embedding I]
    \label{thm:IotaEmbedding1}
    Let $V$ be a Hausdorff barrelled DF-space. Then the mapping
    \begin{equation}
        \label{eq:IotaEmbedding2}
        \iota
        \colon
        \Sym_{0,\epsilon}^\bullet(V)
        \longrightarrow
        \Pol_\beta^\bullet(V'_\beta)
        \subseteq
        \Holomorphic_\beta(V'_\beta)
    \end{equation}
    is a grading preserving linear topological embedding.
\end{theorem}
\begin{proof}
    Let $\seminorm{q} \in \cs(V)$ and consider the bounded polar $B'
    \coloneqq \Ball_{\seminorm{q},1}(0)^\polar$ from
    \eqref{eq:PolarOfCylinder}. Let moreover
    \begin{equation*}
        v
        =
        \sum_{n=0}^{\infty} v_n
        \in
        \Sym^\bullet_{0,\epsilon}(V)
        \qquad
        \textrm{with homogeneous components} \quad
        v_n \in \Sym_\epsilon^n(V)
    \end{equation*}
    and write $P_n \coloneqq \iota_n(v_n) \in \Pol^n(V'_\beta)$. Using
    the polarization estimate~\eqref{eq:PolarizationEstimate} for
    $B'$, which is indeed absolutely convex as a polar \eqref{eq:Polar}
    by the triangle inequality, leads to
    \begin{equation*}
        \seminorm{q}^{\tensor_\epsilon n}(v_n)
        =
        \sup_{v_1', \ldots, v_n' \in B'}
        \abs[\Big]
        {
            \bigl(
                v_1' \tensor \cdots \tensor v_n'
            \bigr)(v_n)
        }
        =
        \sup_{v_1', \ldots, v_n' \in B'}
        \abs[\Big]
        {
            \widecheck{P}_n
            \bigl(
                v_1', \ldots, v_n'
            \bigr)
        }
        \le
        \frac{n^n}{n!}
        \cdot
        \seminorm{p}_{B'}(P_n)
    \end{equation*}
    for $n \in \N$. Hence, in view of $\abs{v_0} =
    \seminorm{p}_{B'}(P_0)$ and $n^n/n! \le \E^n$, we have
    \begin{equation*}
        \seminorm{q}_{c}^{(0,\epsilon)}(v)
        =
        \sum_{n=0}^{\infty}
        c^n
        \cdot
        \seminorm{q}^{\tensor_\epsilon n}(v_n)
        \le
        \sum_{n=0}^{\infty}
        \frac{(c \cdot n)^n}{n!}
        \cdot
        \seminorm{p}_{B'}(P_n)
        \le
        \seminorm{r}_{c \cdot \E \cdot B',\abs{\argument},0}
        \bigl(
        \iota(v)
        \bigr)
    \end{equation*}
    with the continuous seminorm $\seminorm{r}_{c \cdot \E \cdot
    B',\abs{\argument},0} \in \cs(\Pol_\beta(V'_\beta))$ from
    \eqref{eq:BoundedAlternativeSeminorm} and any $c \ge 0$.
\end{proof}

\begin{remark}[Reflexivity]
    Let $V$ be a Hausdorff barrelled DF-space. We have found a
    conceptual description of the locally convex topology of
    $\Sym_{0,\epsilon}^\bullet(V)$: it is simply the topology of uniform
    convergence on bounded sets for the corresponding polynomials.
    Hence, our function-theoretic results from
    Section~\ref{sec:SpacesOfBoundedFunctions} apply. We have shown the
    equivalence of global data on bounded sets and Taylor data at a
    single point, encoded as the homogeneous components of the tensor,
    see again Corollary~\ref{cor:TensorAsFrechet}. The natural
    question then concerns the size of the closures
    \begin{equation}
        \widehat{\iota}_n
        \bigl(
            \widehat{\Sym}_{\epsilon}^n(V)
        \bigr)
        \subseteq
        \Pol^n_\beta(V'_\beta)
        \qquad \textrm{and} \qquad
        \widehat{\iota}
        \bigl(
            \widehat{\Sym}_{0,\epsilon}^\bullet(V)
        \bigr)
        \subseteq
        \Holomorphic_\beta(V'_\beta).
    \end{equation}
    For $n=1$ this recovers the notion of strong reflexivity and thus
    the surjectivity of $\widehat{\iota_n}$ for~$n\ge2$ may be seen as
    higher order reflexivity. We leave this question for posterity.
\end{remark}

However, for the purposes of strict deformation quantization,
Theorem~\ref{thm:IotaEmbedding1} is not quite satisfactory, as the
former's techniques rely on the universal property of the projective
tensor product. The final ingredient we need to guarantee that $\iota$
is an embedding also for the $\Sym_{0,\pi}$-topology turns out to be
nuclearity. Indeed, on tensor products with nuclear spaces, the
locally convex topologies generated by projective and injective tensor
products coincide, i.e. \eqref{eq:ProjectiveVsInjectiveIdentity}
constitutes a homeomorphism.
\begin{theorem}[Embedding II]
    \label{thm:IotaEmbedding2}%
    Let $V$ be a nuclear Hausdorff barrelled DF-space. Then the mapping
    \begin{equation}
        \label{eq:IotaEmbedding}
        \iota
        \colon
        \Sym_{0,\pi}^\bullet(V)
        \longrightarrow
        \Pol_\beta^\bullet(V'_\beta)
    \end{equation}
    is a grading preserving linear topological embedding.
\end{theorem}
\begin{proof}
    Let $\seminorm{q} \in \cs(V)$. By nuclearity of $V$, there exists
    another seminorm $\seminorm{q}' \in \cs(V)$ with $\seminorm{q} \le
    \seminorm{q}'$ such that
    \begin{equation*}
        \seminorm{q} \tensor_\pi \seminorm{m}
        \le
        \nu
        \cdot
        \bigl(
            \seminorm{q} \tensor_\epsilon \seminorm{m}
        \bigr)
        \qquad
        \textrm{for all }
        \seminorm{m} \in \cs(V),
    \end{equation*}
    where $\nu \ge 0$ is the nuclear norm of the canonical mapping
    $V_{\seminorm{q}'} \longrightarrow V_{\seminorm{q}}$
    between the local Banach spaces, see \cite[Thm.~6.38]{vogt:2000a}.
    For tensor powers, this implies
    \begin{equation*}
        \seminorm{q}^{\tensor_\pi n}
        \le
        \nu^{n-1}
        \cdot
        \seminorm{q}^{\tensor_\epsilon n}
        \qquad
        \textrm{for }
        n \in \N.
    \end{equation*}
    This scales polynomially, and thus we may proceed as in the proof of
    Theorem~\ref{thm:IotaEmbedding1} to derive the desired continuity
    estimate after passing to injective tensor powers.
\end{proof}

This explains why the nuclearity in Example~\ref{ex:NuclearFrechet} was
a rather desirable property. Our considerations also match nicely with
\cite[Thm.~4.10]{waldmann:2014a}, which asserts that $\Sym_{R,\pi}(V)$
is nuclear iff $V$ is, where $R \ge 0$ as usual.
\begin{corollary}
    Let $V$ be a nuclear Hausdorff locally convex space. Then
    \begin{equation}
        \Sym_{R,\pi}^\bullet(V)
        \cong
        \Sym_{R,\epsilon}^\bullet(V)
    \end{equation}
    as locally convex algebras.
\end{corollary}